\documentclass[12pt,oneside,reqno]{amsart}
\usepackage{comment}
\usepackage{times}
\usepackage[T1]{fontenc}
\usepackage[utf8]{inputenc}
\usepackage{lmodern}
\usepackage[english]{babel}
\usepackage{geometry}
\usepackage{relsize}
\usepackage{color}
\usepackage{booktabs}
\usepackage{mathtools}
\usepackage{enumitem}
\usepackage{amsthm}
\usepackage{amsmath,amssymb}
\usepackage{emptypage}
\usepackage{xfrac}    
\usepackage{faktor}
\usepackage{graphicx}
\usepackage{listings}
\usepackage{lipsum}
\usepackage{setspace}
\usepackage{lipsum}
\usepackage{newlfont}
\usepackage[mathscr]{eucal}	
\usepackage{physics}
\usepackage{tikz}
\usepackage{url}
\usepackage{hyperref}
\usetikzlibrary{positioning}
\usepackage{xcolor}
\usepackage[new]{old-arrows}
\usepackage{multirow}
\theoremstyle{plain}
\newtheorem{theorem}{Theorem}[section]

\newtheorem{proposition}[theorem]{Proposition}
\newtheorem{corollary}[theorem]{Corollary}

\theoremstyle{definition}

\newtheorem{example}[theorem]{Example}

\newtheorem{remark}{Remark}[section]
\usepackage{setspace}
\usepackage[signatures,write]{frontespizio}

\usepackage{fancyhdr}

\newcommand{\Z}{\mathbb{Z}} % \Z   = numeri relativi
\newcommand{\R}{\mathbb{R}} % \R   = numeri reali
\newcommand{\C}{\mathbb{C}} % \C   = numeri complessi
\newcommand{\n}{\mathfrak{n}}
\newcommand{\g}{\mathfrak{g}}
\newcommand{\s}{\mathfrak{s}}
\newcommand{\h}{\mathfrak{h}}

\AtBeginDocument
{
\setlength\abovedisplayskip{.6\baselineskip}
\setlength\abovedisplayshortskip{.6\baselineskip}
\setlength\belowdisplayskip{.5\baselineskip}
\setlength\belowdisplayshortskip{.5\baselineskip}
}
\begin{document}

\title{ SKT solvable Lie algebras with codimension two nilradical} 

\author{Beatrice Brienza}
\address[Beatrice Brienza]{Dipartimento di Matematica ``G. Peano'', Universit\`{a} degli studi di Torino \\
Via Carlo Alberto 10\\
10123 Torino, Italy}
\email{beatrice.brienza@unito.it}

\author{Anna Fino}
\address[Anna Fino]{Dipartimento di Matematica ``G. Peano'', Universit\`{a} degli studi di Torino \\
Via Carlo Alberto 10\\
10123 Torino, Italy\\
\& Department of Mathematics and Statistics, Florida International University\\
Miami, FL 33199, United States}
\email{annamaria.fino@unito.it, afino@fiu.edu}

\keywords{Solvable Lie algebra, SKT metric, Generalized K\"ahler structure}

\subjclass[2010]{53C55; 53C05; 53C30; 53C44}

\maketitle

\begin{abstract}  In the present paper we study SKT and generalized K\"ahler structures on solvable Lie algebras with (not necessarily abelian) codimension  two nilradical. We treat separately the case of $J$-invariant nilradical and non $J$-invariant nilradical. A classification of such SKT Lie algebras in dimension six is provided. In particular, we give a general  construction to extend SKT nilpotent Lie algebras to SKT solvable Lie algebras of higher dimension, and we construct new examples of SKT and generalized K\"ahler compact solvmanifolds.  

\end{abstract}

\section{Introduction}
Let $(M,  J, g)$ be an Hermitian manifold of complex dimension $n$ with fundamental form $\omega=g(J\cdot,\cdot)$. A connection $\nabla$ on $TM$ is said to be \emph{Hermitian} if $\nabla g=0$ and $\nabla J=0$. In \cite{PG}, Gauduchon has introduced an affine line of Hermitian connections, known as \emph{Gauduchon} or \emph{canonical} connections, which can be written as
\begin{equation*} 
g(\nabla^t_XY,Z)=g(\nabla^{LC}_XY,Z)+\frac{t-1}{4}(d^c\omega)(X,Y,Z)+\frac{t+1}{4}(d^c\omega)(X,JY,JZ),
\end{equation*}
where $d^c \omega=-Jd\omega$. We adopt the convention $Jd\omega(X,Y,Z):=d\omega(JX,JY,JZ)$.
When the manifold $(M,J,g)$ is K\"ahler, namely, $d^c\omega$ is zero, the line  of connections collapses to a single point, which is the Levi-Civita connection $\nabla^{LC}$, whereas, when  $(M,J, g)$ is non-K\"ahler, the line is non trivial and the connections $\nabla^t$ have non vanishing torsion. 
For particular values of $t \in \R$, we recover connections that play a relevant role in complex (non-K\"ahler) geometry. For $t=1$, the \emph{Chern connection} $\nabla^{1}=\nabla^{Ch}$ (\cite{CH}), characterized by $(\nabla^{Ch})^{0,1}=\overline{\partial}$, and for $t=-1$ the \emph{Bismut} (or \emph{ Strominger}) \emph{connection} $\nabla^{-1}=\nabla^{B}$ (\cite{Bi,Strominger}), characterized as the only Hermitian connection with totally skew-symmetric torsion. Although $\nabla^{LC}, \nabla^{B}, \nabla^{Ch}$ are mutually different connections, any one of them completely determines the other two, e.g. the Bismut connection can be defined in terms of the Levi-Civita one as
\begin{equation*} 
g(\nabla^B_XY,Z)=g(\nabla^{LC}_XY,Z)-\frac{1}{2}d^c\omega(X,Y,Z),
\end{equation*}
and its torsion $3$-form $c$, also called the \emph{Bismut torsion}, is
$$
c(X, Y,  Z) = g(T^B (X, Y), Z) = d \omega (JX, JY, JZ) = - d^c \omega (X, Y, Z).
$$
If the torsion $3$-form $c$ is closed, i.e.,  $d d^c \omega =0$ or, equivalently, $\partial \overline \partial \omega=0$, the metric $g$ is said {\em strong K\"ahler with torsion} (SKT in short) or {\em pluriclosed}. \\
SKT metrics also appear naturally in the setting of \emph{generalized K\"ahler geometry}: according to \cite{Gualtieri}, a \emph{generalized K\"ahler} structure on a $2n$-dimensional manifold $M$ is a pair of commuting generalized complex structures $(\cal{J}_1,\cal{J}_2)$ such that $\cal{G}=-\cal{J}_1\cal{J}_2$ is a positive definite metric on $TM \oplus T^*M$. It turns out  that the generalized K\"ahler condition can be equivalently described in the language of Hermitian geometry as a bi-Hermitian structure $(J_+,J_-, g)$ on $M$ such that the Bismut torsions $c_\pm$ of $(J_\pm, g)$ satisfy $c_+=-c_-$ and $dc_+=0$ (see \cite{Gualtieri} for further details).  In \cite{Hitchin}, Hitchin has proved that whenever $[J_+,J_-]\neq 0$, the tensor $\sigma=[J_+,J_-]g^{-1}$ defines a holomorphic Poisson structure. Generalized K\"ahler structures $(g,J_\pm)$ are said  \emph{non-split} if $[J_+,J_-]\neq 0$ and \emph{split} otherwise: in the latter case the tensor $Q=J_+J_-$ is an involution of $TM$ and induces the splitting $TM=T_+M \oplus T_-M$ in  terms of  its $\pm1$-eigenbundles. \\
As one may easily observe, if $(M, J, g)$ is K\"ahler, then $(M,  \pm J, g)$ is generalized K\"ahler: generalized K\"ahler structures which arise in such a way are said \emph{trivial}. As a consequence, much study has been devoted to the explicit construction of non-trivial generalized K\"ahler manifolds, e.g., \cite{ADE,AGG, AG, BM, CG, DM, FP, FP2, FT, BF}.  In contrast with the case of compact nilmanifolds, which cannot admit  invariant generalized K\"ahler structures unless they are tori \cite{CA}, several families of (non-K\"ahler) compact generalized K\"ahler manifolds are  compact solvmanifolds (\cite{ FP, FP2, FT}), namely, compact quotients of a connected and simply connected solvable Lie group by a lattice.  \\
By \cite{Has}, a compact solvmanifold admits a K\"ahler structure if and only if it is a finite quotient of a complex torus which has the structure of a complex torus bundle over a complex torus. No general restrictions are known to the existence of generalized K\"ahler structures on compact solvmanifolds: up to now the only known examples have abelian nilradical, even though it is still an open question whether this is true in general. \\
Since the underlying metric of a generalized K\"ahler structure is in particular SKT, a more general problem regards the existence of SKT structures on solvmanifolds $\Gamma \backslash G$.  When the complex structure is invariant, i.e., it descends from a left-invariant complex structure on $G$, exploiting the symmetrization process (\cite{BE}) the problem reduces to investigate the existence of SKT inner products at the level of the Lie algebra $\g=Lie(G)$. Although it is a simplified setting, the solvable case seems to be harder than the nilpotent case, even in low dimensions. The existence of a SKT structure on a nilpotent Lie algebra imposes severe restrictions: Arroyo and Nicolini proved in  \cite{ArNic} that the existence of a SKT metric on a nilpotent Lie algebra implies that the nilpotency step is at most $2$, as conjectured in \cite{EFV}. Furthermore, SKT nilpotent Lie algebras of dimension 6 and 8  have been fully classified  in \cite{FPS04} and  \cite{EFV}, respectively. \\
SKT structures on (non-nilpotent) solvable Lie algebras have been instead studied in several papers \cite{FP,FP2,FP3,AL,FS,MS}, however, a full classification has been obtained only in dimension $4$ in \cite{MS}. In dimension $6$, the second author and Paradiso have classified the SKT almost nilpotent Lie algebras in \cite{FP,FP2,FP3} and Freibert and Swann have classified in \cite{FS,FS2} the SKT $2$-step solvable Lie algebras.  Furthermore, Hermitian geometry of solvable Lie algebras with an abelian ideal of codimension $2$ have been recently investigated in \cite{CZ,GZ}.   \\
Another special class of Hermitian structures on complex manifolds is provided by the \emph{balanced} structures, namely, Hermitian structures $(g,J)$ whose fundamental form $\omega$ is co-closed or, equivalently, satisfying $d\omega^{n-1}=0$. It has been conjectured in \cite{FV}, that a compact complex manifold cannot admit both SKT and balanced metrics, unless it admits K\"ahler metrics as well. In the locally homogeneous setting, the conjecture holds true for nilmanifolds (\cite{FV}), almost abelian solvmanifolds (\cite{FP4}), six-dimensional almost nilpotent solvmanifolds (\cite{FP3}) and on solvable Lie algebras with an abelian ideal of codimension $2$ (\cite{CZ,GZ,FS2}). Non compact counterexamples of this conjecture have been instead constructed in \cite{FS2}. \\
In this work, we mainly focus on Hermitian structures on solvable Lie algebras $\g$ with (not necessarily abelian) nilradical $\h$ of codimension $2$. Since $\h$ has even dimension, if $(\g, J, \langle \cdot,\cdot \rangle)$ is Hermitian, we may distinguish the two cases $J\h=\h$ and 
 $J\h \neq \h$,
which will be treated separately. 
Section \ref{section2}  is devoted to study the first case, i.e.,  Hermitian Lie algebras $(\g, J, \langle \cdot,\cdot \rangle)$ with codimension $2$ $J$-invariant nilradical $\h$. In Theorem \ref{theorem:sktinvariant} we prove that  the Hermitian structure $(J, \langle \cdot,\cdot \rangle)$ is always Chern Ricci flat, extending a result given in \cite{GZ} in the case of $\h$ being abelian, and we give necessary conditions for  $(J, \langle \cdot,\cdot \rangle)$ to be SKT. In particular we observe that the SKT condition imposes restrictions on the structure of $\h$, namely, it has to be $2$-step nilpotent. As a consequence (see Corollary \ref{thm:abelian}) we study the SKT condition when $\h$ is abelian.  Although in \cite{GZ} Hermitian structures on Lie algebras with a $J$-invariant abelian ideal of codimension $2$ are investigated, we are able to prove that when the ideal coincides with the nilradical of the Lie algebra, the SKT condition has a more specialized characterization. Furthermore, we prove that if $\g$ admits a generalized K\"ahler structure $( J_{\pm}, \langle \cdot,\cdot \rangle)$ satisfiying $J_\pm \h=\h$, then  $(J_\pm, \langle \cdot,\cdot \rangle)$ must be K\"ahler.  \\
In  Section \ref{section3} we investigate Hermitian Lie algebras $(\g, \langle \cdot,\cdot \rangle, J)$ with codimension $2$ nilradical $\h$ such that $J\h \neq \h$, with a special focus on the existence  of generalized K\"ahler structures  $(J, I, \langle \cdot,\cdot \rangle)$ such that $J\h=\h$ and $I\h \neq \h$. 
We also construct the first example of generalized K\"ahler Lie algebra with non-abelian nilradical: however, the example is not unimodular.
 In Section \ref{section4}, we provide a full classification of unimodular six-dimensional solvable Lie algebras with codimension $2$ nilradical. 
Section  \ref{section5} is devoted to the construction of new examples of SKT solvable Lie algebras. In particular, we provide a general process to extend SKT nilpotent Lie algebras of dimension $2n$ to SKT solvable Lie algebras of dimension $2n+2k$. We apply this construction to the  six-dimensional  SKT nilpotent Lie algebras classified  in \cite{FPS04} (see also \cite{UG}), to obtain new families of SKT solvable Lie algebras in dimension eight. 
  Finally, in the last section, we exhibit some results on the existence of generalized K\"ahler structures on solvmanifolds with codimension  two  nilradical and we construct new examples of compact SKT and generalized K\"ahler solvmanifolds.
 
 \smallskip

\textbf{Acknowledgements}. 
The authors are partially supported by Project PRIN 2022 \lq \lq Geometry and Holomorphic Dynamics” and by GNSAGA (Indam). Anna Fino   is also supported  by a grant from the Simons Foundation (\#944448). The authors would like to thank Alejandro Tolcachier for his interest in this paper. The first author would like to thank Tommaso Sferruzza for useful conversations and remarks. 
\smallskip

\section{ Case $J \h = \h$} \label{section2}
Let $\g$ be a $2n$-dimensional solvable Lie algebra with a codimension $2$ nilradical $\h$, endowed with an almost Hermitian structure $(J, \langle \cdot, \cdot \rangle)$ such that $J\h=\h$. We can decompose $\g$ as the orthogonal sum $\h \oplus \h^\perp$, where each summand is $J$-invariant and $\dim_{\C} \h^{\perp} =1$. \\
Let  $U$   be a unit vector of $\h^\perp$.  We have that the Lie bracket of $\g$ is given by
\[
[U, Y]=A Y, \ [JU ,Y]=B Y, \  [Y,W]=  \mu (Y, W),  \  [U,JU]=V,    \quad \forall Y,W \in \h,
\]
where $A:= ad_U \vert_{\frak h}, B := ad_{JU} \vert_{\frak h}$ are derivations of $\h$ and  $\mu$  is the  Lie bracket $[\cdot, \cdot]_{\h}$ on $\frak h. $ Note that $V \in \frak h$, since $\frak g^1$ is contained in $\h$ and  $ad_{V}\vert_{\h} = [A, B]$. \\
Furthermore, since  the restrictions of $J$ and $\langle \cdot, \cdot \rangle$ to $\h^\perp$ are completely determined by choosing the orthonormal basis $\{U,JU\}$ of $(\frak h^{\perp}, \langle \cdot, \cdot \rangle_{\frak h})$, we have that  the almost Hermitian Lie algebra $(\g, J, \langle \cdot, \cdot \rangle)$ is uniquely determined by the algebraic data
$(U, JU, A,B,V,\mu, J_\h, \langle \cdot,\cdot \rangle_\h ).$
\begin{remark}  
Note that the data $(U, JU, A,B,V,\mu)$ determine a Lie algebra if and only if $A,B$ are derivations of $\h$, $[A,B]=ad_V \vert_\h$ and $\mu$ is a Lie bracket on $\h$.
\end{remark}
\begin{theorem} \label{theorem:sktinvariant}
Let  $(\g,  J, \langle \cdot, \cdot \rangle)$ be an  almost Hermitian solvable Lie algebra with a $J$-invariant codimension $2$ nilradical $\frak h$ and let $\{U, JU \}$ be an orthonormal basis of $(\frak h^{\perp},  \langle \cdot, \cdot \rangle_{ \h^{\perp}})$. Then 
\begin{enumerate}
\item[(i)] The complex structure  $J$ is integrable, if and only if  $J_\h$ is a complex structure on $\h$ and  $A:= ad_U \vert_{\frak h}$ and  $B := ad_{JU} \vert_{\frak h}$ satisfy the following condition  $$[J_\h,A]J_\h+[J_\h,B]=0.$$ \label{thm:integrability}
\item[(ii)]  If $J$ is integrable, then  the Hermitian Lie algebra  $(\g,  J, \langle \cdot, \cdot \rangle)$ is Chern Ricci flat. \label{thm:chernricciflat}
\item[(iii)] If $J$ is integrable, $\g$ is unimodular and $[U,JU]=0$, then   $(J, \langle \cdot, \cdot \rangle)$ is balanced if and only if $(J_\h, \langle \cdot, \cdot \rangle_\h)$ is balanced. 
\item [(iv)] If  $J$ is integrable and  the Hermitian Lie algebra  $(\g, J,  \langle \cdot, \cdot \rangle)$ is SKT, then the nilradical  $\frak h$ is at most $2$-step nilpotent and  the restrictions $A_{\mathfrak{z}(\h)}, B_{\mathfrak{z}(\h)}$ of $A$ and $B$ to the center  $\mathfrak{z}(\h)$ of $\frak h$ must satisfy the following conditions 
$$
A_{\mathfrak{z}(\h)}, B_{\mathfrak{z}(\h)} \in \mathfrak{so}(\mathfrak{z}(\h)), \quad [A_{\mathfrak{z}(\h)},J_{\mathfrak{z}(\h)}]=[B_{\mathfrak{z}(\h)},J_{\mathfrak{z}(\h)}]=0. 
$$
\end{enumerate}
\end{theorem}
\begin{proof}
Using that the Nijenhuis tensor $N$ of $J$ satisfies the condition $N(J \cdot,J \cdot)=-N(\cdot,\cdot)$, one can see that it is enough to check the vanishing of $N(U,Y)$ and $N(Y,Z)$ for  every $Y,Z \in \h$. By a direct computation we have that $N(Y,Z) = N_{J_{\h}}(Y,Z)$ and 
$$
N(U,Y)= [U,Y]+J_\h([JU,Y]+[X,J_\h Y])-[JU,J_\h Y]
= (A+J_\h AJ_\h+J_\h B-BJ_\h) Y,
$$
from which \emph{(i)} follows. To prove \emph{(ii)}, following \cite{VZ} (see also \cite[Formula 5.8] {FP}), we use that the Ricci form of the Chern Connection $\rho^{Ch}$ is given by $d\eta^{Ch}$, where 
\[
\eta^{Ch}(Y)=\frac{1}{2} \big( \tr(ad_Y \circ J)- \tr ad_{JY} \big),  \ \forall Y \in \g.
\]
If $Y \in \h$, then $\eta^{Ch}(Y)=\eta^{Ch}_{\h}(Y)=0$, since $\h$ is nilpotent (a proof can be found in \cite[Proposition 2.1]{LRV}).
Furthermore, since
\[
\begin{split}
\eta^{Ch}(U)&=\frac{1}{2} \big( \tr(ad_U \circ J)- \tr ad_{JU} \big)=\frac{1}{2} \big( \tr(A \circ J_\h)- \tr B \big)\\
\eta^{Ch}(JU)&=\frac{1}{2} \big( \tr(ad_{JU} \circ J)+ \tr ad_{U} \big)=\frac{1}{2} \big( \tr(B \circ J_\h)+ \tr A \big), \\
\end{split}
\]
we obtain
\[
\eta^{Ch}=\frac{1}{2} \big( \tr(A \circ J_\h)- \tr B \big) u + \frac{1}{2} \big( \tr(B \circ J_\h)+ \tr A \big) Ju,
\]
where  $\{ u,  Ju \}$  is the dual basis of $\{ U,JU \}$.  Then \emph{(ii)}  follows by differentiating $\eta^{Ch}$, since $du=d(Ju)=0$.  \\
To prove \emph{(iii)}, we observe that the orthogonal splitting $\g=\h \oplus \h^\perp$, implies that the fundamental form $\omega$ of $(J, \langle \cdot, \cdot \rangle)$ can be written as $\omega=\omega_\h + u \wedge Ju$. Its $(n-1)$ power is then given by $\omega^{n-1}=\omega_\h^{n-1}+(n-1) \omega^{n-2}\wedge u \wedge Ju$. We first observe that  for any $\alpha \in {\bigwedge}^k \h^*$, 
\begin{equation} \label{eqn:differential}
d\alpha=u \wedge A^* \alpha + Ju \wedge B^* \alpha - u \wedge  Ju \wedge \iota_V \alpha+d_\h \alpha,
\end{equation}
where 
\[C^*\gamma=-[\gamma(C\cdot,\dots,\cdot)+\dots+\gamma(\cdot,\dots,C\cdot)], \quad \forall \gamma \in {\bigwedge}^k \h^*, \, C \in \mathfrak{gl} (\h),\]
and $d_\h$ stands for the exterior differential of the nilpotent Lie algebra $(\h, \mu)$.\\
Exploiting that $du=dJu=0$ and that $V=[U,JU]=0$, we get
\[
d\omega^{n-1}=d \omega_\h^{n-1}+(n-1) d_\h \omega^{n-2}\wedge u \wedge Ju.
\] 
We claim that $d \omega_\h^{n-1}=0$. Assume by contradiction that $d\omega_\h^{n-1}=u \wedge A^* \omega_\h^{n-1}+ Ju \wedge B^* \omega_\h^{n-1} \neq 0$, hence, at least one between  $A^* \omega_\h^{n-1}$ and $B^* \omega_\h^{n-1} $ is non-zero. If the first holds, then one may consider $d (\omega_\h^{n-1} \wedge Ju)= u \wedge A^* \omega_\h^{n-1} \wedge Ju \neq 0$. Moreover, this provides a contradiction: since $\g$ is unimodular, any $2n-1$ form on $\g$ is closed. The other case proceeds in the same way by considering  $d (\omega_\h^{n-1} \wedge u)$. Therefore, $d\omega^{n-1}=(n-1) d_\h \omega^{n-2}\wedge u \wedge Ju$ is zero if and only if $d_\h \omega^{n-2}=0$.\\
 The first part of \emph{(iv)} follows from \cite{ArNic}. 
Indeed,  if the Hermitian Lie algebra $(\g, J, \langle \cdot, \cdot \rangle)$ is SKT, then $(\h, J_\h,  \langle \cdot, \cdot \rangle_\h)$ is SKT (see \cite [Proposition 3.1]{ArNic}). 
In particular,  $(\h, J_\h,  \langle \cdot, \cdot \rangle_\h)$ is at most  $2$-step nilpotent (\cite[Theorem 4.8]{ArNic}).  Since, $(\h, \langle \cdot,\cdot \rangle_\h, J_\h)$ is SKT and at most $2$-step nilpotent we have the orthogonal decomposition $\h =\mathfrak{z}(\h) \oplus \mathfrak{z}(\h)^\perp$, where $\mathfrak{z} (\h)$ is the center of $\h$. Observe that each summand in the decomposition $\h=\mathfrak{z}(\h) \oplus \mathfrak{z}(\h)^\perp$ is $J_\h$-invariant by \cite[Proposition 3.5]{EFV}. \\
Hence, $J_\h$ is determined by $J_{\mathfrak{z}(\h)}$ and $J_{\mathfrak{z}(\h)^\perp}$. Furthermore, since $A,B \in Der(\h)$, then they must preserve the center. Indeed, if $Z \in \mathfrak{z}(\h)$, then 
\[
0=A \big(\mu(Z,Y)\big)=\mu(AZ,Y)+\mu(Z,AY)=\mu(AZ,Y), \ \forall Y \in \h,
\] 
from which follows that $AZ \in \mathfrak{z}(\h)$, and analogously for $B$. With respect to the decomposition $\h=\mathfrak{z}(\h) \oplus \mathfrak{z}(\h)^\perp$ we have that 
\[
A=\begin{pmatrix} 
A_{\mathfrak{z}(\h)} & \ast_A \\
0 & A_{\mathfrak{z}(\h)^\perp}
\end{pmatrix} \ \ \text{and} \ \ 
B=\begin{pmatrix}
B_{\mathfrak{z}(\h)} & \ast_B \\
0 & B_{\mathfrak{z}(\h)^\perp}
\end{pmatrix}.
\]
In particular, $\mathfrak{z}(\h)$ is a $J$-invariant ideal of $\g$ and the integrability condition involving $A,B$ and $J_\h$ reads $A_{\mathfrak{z}(\h)}+J_{\mathfrak{z}(\h)} A_{\mathfrak{z}(\h)}J_{\mathfrak{z}(\h)}+J_{\mathfrak{z}(\h)} B_{\mathfrak{z}(\h)}-B_{\mathfrak{z}(\h)}J_{\mathfrak{z}(\h)}=0$ on $\mathfrak{z}(\h)$. \\
 Let $c$ be the  Bismut torsion $3$-form of $(J, \langle \cdot,\cdot \rangle)$. Then, for any $Z \in \mathfrak{z}(\h)$  using the formula for the $dc$ in \cite{DF} (see also \cite[Formula 3] {ArNic}) we get 
\begin{align} \label{eqn:expression1} 
dc(JZ,Z,U,JU)=& \norm{ AJZ}^2 +\norm{ BJZ}^2  +\norm{ AZ}^2 +\norm{ BZ}^2 \\  
 &-\langle AJAJZ,Z\rangle-\langle JAJAZ,Z\rangle -\langle BJBJZ,Z\rangle-\langle JBJBZ,Z\rangle. \nonumber 
\end{align}
By  \emph{(i)}, this is equivalent to 
\begin{align*}  
dc(JZ,Z,X,JX)=& \norm{ AJZ}^2 +\norm{ BJZ}^2  +\norm{ AZ}^2 +\norm{ BZ}^2 \\  
 &+\langle \big( 2(A^2+B^2+AJ_\h B -BJ_\h A) -J_\h [A,B] -[A,B] J_\h \big) Z,Z\rangle.
 \end{align*}
Moreover, since $[A,B]=ad_{V} \vert_{ \h}$, and $Z$ is in the center of $\h$, $[A,B] Z=0$. Analogously, since $\mathfrak{z}(\h)$ is $J_\h$ invariant, also $[A,B] J_\h Z=0$. Hence,
\begin{align*} 
dc(JZ,Z,X,JX)=& \norm{ AJZ}^2 +\norm{ BJZ}^2  +\norm{ AZ}^2 +\norm{ BZ}^2 \\  
 &+\langle \big( 2(A^2+B^2+AJ_\h B -BJ_\h A) \big) Z,Z\rangle.
 \end{align*}
 Hence, the SKT condition yields
 \begin{equation} \label{eqn:sktcondition}
 \norm{ AJZ}^2 +\norm{ BJZ}^2  +\norm{ AZ}^2 +\norm{ BZ}^2 +\langle \big( 2(A^2+B^2+AJ_\h B -BJ_\h A) \big) Z,Z\rangle=0
 \end{equation}
 for any $Z \in \mathfrak{z}(\h)$. Let $\{e_1,\dots,e_{2r}\}$ be any orthonormal basis of $\mathfrak{z}(\h)$.  Without loss of generality, we may assume that $Je_{2j-1}=e_{2j}$ for each $j=1,\dots,r$. Using \eqref{eqn:sktcondition}, we get that $\sum_{j=1}^{2r}dc(Je_j,e_j,X,JX)$ is equal to
\begin{equation*}
\begin{split} 
&\sum_{j=1}^{2r} \norm{ AJ e_j}^2 +\norm{ BJ e_j}^2  +\norm{ Ae_j}^2 + \norm{Be_j}^2 +\langle \big( 2(A^2+B^2+AJ_\h B -BJ_\h A) \big) e_j,e_j\rangle= \\
%&\norm{A_{\mathfrak{z}(\h)}J_\h}^2+\norm{B_{\mathfrak{z}(\h)}J_\h}^2+\norm{A_{\mathfrak{z}(\h)}}^2+\norm{B_{\mathfrak{z}(\h)}}^2 +2\big(\trace(A_{\mathfrak{z}(\h)}^2+B_{\mathfrak{z}(\h)}^2)+A_{\mathfrak{z}(\h)}J_{\mathfrak{z}(\h)}B_{\mathfrak{z}(\h)}-B_{\mathfrak{z}(\h)}J_{\mathfrak{z}(\h)}A_{\mathfrak{z}(\h)})\big)= \\
 &2\big(\norm{A_{\mathfrak{z}(\h)}}^2+\norm{B_{\mathfrak{z}(\h)}}^2+\trace(A_{\mathfrak{z}(\h)}^2)+\trace(B_{\mathfrak{z}(\h)}^2) +\trace(A_{\mathfrak{z}(\h)}J_{\mathfrak{z}(\h)}B_{\mathfrak{z}(\h)}-B_{\mathfrak{z}(\h)}J_{\mathfrak{z}(\h)}A_{\mathfrak{z}(\h)})\big).
\end{split}
\end{equation*}
Furthermore, since $[A_{\mathfrak{z}(\h)},B_{\mathfrak{z}(\h)}]=[A,B]_{\mathfrak{z}(\h)}=0$, then $\trace(A_{\mathfrak{z}(\h)}J_{\mathfrak{z}(\h)}B_{\mathfrak{z}(\h)}-B_{\mathfrak{z}(\h)}J_{\mathfrak{z}(\h)}A_{\mathfrak{z}(\h)})=0$.\\
Hence, the SKT condition implies that
\begin{equation}\label{eqn:trace}
\norm{A_{\mathfrak{z}(\h)}}^2+\norm{B_{\mathfrak{z}(\h)}}^2+\trace(A_{\mathfrak{z}(\h)}^2)+\trace(B_{\mathfrak{z}(\h)}^2) =0.
\end{equation}
We claim $\norm{A_{\mathfrak{z}(\h)}}^2+\trace(A_{\mathfrak{z}(\h)}^2)\ge0$ and $\norm{B_{\mathfrak{z}(\h)}}^2+\trace(B_{\mathfrak{z}(\h)}^2)\ge0$. Indeed, 
\[
\norm{A_{\mathfrak{z}(\h)}}^2+\trace(A_{\mathfrak{z}(\h)}^2)=2\mathlarger{\mathlarger{\sum}}_{i=1}^{2r} \abs{a_{ii}}^2+\mathlarger{\mathlarger{\sum}}_{i<j} (\abs{a_{ij}}^2+ \abs{a_{ji}}^2)+ \mathlarger{\mathlarger{\sum}}_{i<j} 2({a_{ij}}\cdot {a_{ji}})=2\mathlarger{\mathlarger{\sum}}_{i=1}^{2r} \abs{a_{ii}}^2+\mathlarger{\mathlarger{\sum}}_{i<j}(a_{ij}+a_{ji})^2.
\]
The same argument holds for $B_{\mathfrak{z}(\h)}$.  
More specifically, \eqref{eqn:trace} implies that if $(\g, J, \langle \cdot,\cdot \rangle)$ is SKT, then $A_{\mathfrak{z}(\h)}$ and $B_{\mathfrak{z}(\h)}$ are skew-symmetric matrices with respect to any orthonormal basis of $(\frak h, J_{\frak h}, \langle \cdot, \cdot \rangle_{\frak h})$. In fact, the choice of the orthonormal basis does not affect the previous computations. This proves that $A_{\mathfrak{z}(\h)}$ and $B_{\mathfrak{z}(\h)}$ are in $\mathfrak{so}(\h)$. \\
With respect to  $\{e_1,\dots,e_{2r}\}$  we may write $A_{\mathfrak{z}(\h)}$ and $B_{\mathfrak{z}(\h)}$ using $2\times2$ block matrices $A_{ij}$ and $B_{ij}$ as follows
\[ 
A_{\mathfrak{z}(\h)}=\begin{pmatrix} 
    A_{1,1} & \dots  & A_{1,r}\\
    \vdots & \ddots & \vdots\\
    A_{r,1} & \dots  & A_{r,r} 
    \end{pmatrix},
  \quad
B_{\mathfrak{z}(\h)}=\begin{pmatrix}
   B_{1,1} & \dots  & B_{1,r}\\
    \vdots & \ddots & \vdots\\
    B_{r,1} & \dots  & B_{r,r} 
\end{pmatrix},
\]
where, for   $ i =1, \ldots, r$,
\begin{align*} 
A_{i,i}=&\begin{pmatrix} 
    0 &  a_{2i-1,2i}\\
    -a_{2i-1,2i} & 0\\
    \end{pmatrix},
\quad 
B_{i,i}=\begin{pmatrix} 
    0 &  b_{2i-1,2i}\\
    -b_{2i-1,2i} & 0\\
    \end{pmatrix},
\end{align*}
and,  for $i<j$, 
\begin{align*} 
A_{i,j}=\begin{pmatrix} 
\begin{array}{ll}
a_{2i-1,2j-1}   &a_{2i-1,2j} \\
a_{2i,2j-1}      &a_{2i,2j}\\
\end{array}
\end{pmatrix}, \ A_{j,i}=-(^t A_{i,j}), \
B_{i,j}=\begin{pmatrix} 
\begin{array}{ll}
b_{2i-1,2j-1}   &b_{2i-1,2j} \\
b_{2i,2j-1}      &b_{2i,2j}\\
\end{array}
 \end{pmatrix}, \ B_{j,i}=-(^tB_{i,j}).
  \end{align*}
Moreover, since we choose $\{e_1,\dots, e_{2r}\}$ satisfying $Je_{2j-1}=e_{2j}$ for each $j=1,\dots,r$,  with respect to such a basis the complex structure $J_{\mathfrak{z}(\h)}$ can be represented by the diagonal block matrix $\operatorname{diag}(\Lambda_1,\dots,\Lambda_{r})$ with
\[ 
\Lambda_i=\begin{pmatrix} 
    0 &  -1\\
    1& 0\\
    \end{pmatrix}.
\]
If one imposes the integrability condition $A_{\mathfrak{z}(\h)}+J_{\mathfrak{z}(\h)} A_{\mathfrak{z}(\h)}J_{\mathfrak{z}(\h)}+J_{\mathfrak{z}(\h)} B_{\mathfrak{z}(\h)}-B_{\mathfrak{z}(\h)}J_{\mathfrak{z}(\h)}=0$ then the following linear conditions hold:
\begin{equation}\label{eqn:integrability_cond}
\begin{cases}
a_{2i-1,2j-1}=a_{2i,2j}+b_{2i,2j-1}+b_{2i-1,2j}\\
b_{2i,2j}=a_{2i-1,2j}+a_{2i,2j-1}+b_{2i-1,2j-1}.\\
\end{cases}
\end{equation}
We compute the components $dc(e_{2j-1},Je_{2j-1},X,JX)$, using \eqref{eqn:expression1}. With a  straightforward computation one gets
\begin{equation}\label{eqn:expression2}
\begin{split}
dc(e_{2j-1},Je_{2j-1},X,JX)= &- 2  \mathlarger{\mathlarger{\sum}}_{k=1}^{j-1} (b_{2k,2j-1}+b_{2k-1,2j})^2+(a_{2k-1,2j}+a_{2k,2j-1})^2 \\
& -2 \mathlarger{\mathlarger{\sum}}_{k=j+1}^{r-1} (b_{2j,2k-1}+b_{2j-1,2k})^2+(a_{2j-1,2k}+a_{2j,2k-1})^2.
\end{split}
\end{equation}
Repeating the same argument for each $j=1,\dots,r$, the vanishing of \eqref{eqn:expression2} leads to the identities $b_{2i,2j-1}=-b_{2i-1,2j}$ and 
$a_{2i,2j-1}=-a_{2i-1,2j}.$
Moreover, plugging these identities in \eqref{eqn:integrability_cond} we get
\[
\begin{cases}
a_{2i,2j}=a_{2i-1,2j-1}\\
b_{2i,2j}=b_{2i-1,2j-1}.\\
\end{cases}
\]
Hence we have that for any $i < j$,  the matrices $A_{ij}$ and $B_{ij}$ are of the kind
\[ 
A_{ij}=\begin{pmatrix}
a_{2i-1,2j-1} & a_{2i-1,2j} \\
-a_{2i-1,2j} &a_{2i-1,2j-1} \\
\end{pmatrix}, \
B_{ij}=\begin{pmatrix}
b_{2i-1,2j-1} & b_{2i-1,2j} \\
-b_{2i-1,2j} &b_{2i-1,2j-1} \\
\end{pmatrix},
\] 
and it is straightforward to observe that $[A_{\mathfrak{z}(\h)},J_{\mathfrak{z}(\h)}]=[B_{\mathfrak{z}(\h)},J_{\mathfrak{z}(\h)}]=0$. 
\end{proof}
\begin{remark}  If the Hermitian Lie algebra $(\g, \langle \cdot, \cdot \rangle, J)$ is K\"ahler, then $\h$ is abelian. Exploiting \eqref{eqn:differential}, since the fundamental form $\omega$ splits as the sum $\omega_\h + u \wedge Ju$, we get that if  $(\langle \cdot, \cdot \rangle, J)$ is K\"ahler, then
$0=d \omega= u \wedge A^* \omega_\h  + Ju \wedge B^* \omega_\h  - u\wedge  Ju \wedge \iota_V \omega_\h +d_\h \omega_\h$, which implies that $d_\h \omega_\h=0$. Moreover, since $\h$ is nilpotent, then it must be abelian. 
\end{remark}
\begin{corollary} 
Let $\g$ be a unimodular solvable Lie algebra with codimension $2$ nilradical $\h$. Assume that $\g$ is endowed with a complex structure $J$ such that $J\h=\h$. If $\g$ admits a $J$-Hermitian SKT metric $\langle \cdot,	\cdot \rangle_1$ and a $J$-Hermitian balanced metric $\langle \cdot,	\cdot \rangle_2$ such that $[\h^{\perp_{\langle \cdot,	\cdot \rangle_2}}, \h^{\perp_{\langle \cdot,	\cdot \rangle_2}}]=0$, then $\g$ admits also a K\"ahler metric. 
\end{corollary}
\begin{proof}
Since $\g$ is unimodular and $[\h^{\perp_{\langle \cdot,\cdot \rangle_2}}, \h^{\perp_{\langle \cdot,	\cdot \rangle_2}}]=0$, by Theorem \ref{theorem:sktinvariant} statement \emph{(iii)} we have that $(\h,J)$ admits a balanced metric. Furthermore, by the proof of statement \emph{(iv)} in Theorem  \ref{theorem:sktinvariant}, $(\h, J)$ admits also a SKT metric, and so $\h$ it is abelian by  \cite{FV}. Since $\g$  contains an abelian ideal of codimension $2$ which is $J$-invariant, then $(\g, J)$ admits also a K\"ahler metric by \cite{CZ, GZ}. 
\end{proof}
As a corollary of Theorem \ref{theorem:sktinvariant} we focus now on the case of $\h$ being abelian. Observe that in this case $J_\h$ is trivially integrable, so the integrability of $J$ simply reduces to  the condition $A+J_\h AJ_\h+J_\h B-BJ_\h=0$.  In \cite{GZ}, Hermitian structures on Lie algebras with a $J$-invariant abelian ideal of codimension $2$ are investigated. The next Corollary shows that when the ideal coincides with the nilradical of the Lie algebra, the SKT condition has a more specialized characterization.
\begin{corollary} \label{thm:abelian}
Let  $(\g,  J, \langle \cdot, \cdot \rangle)$ be an  almost Hermitian solvable Lie algebra  with  a  $J$-invariant codimension 2  abelian nilradical $\frak h$ and let $\{U, JU \}$ be an orthonormal basis of $(\frak h^{\perp},  \langle \cdot, \cdot \rangle_{ \h^{\perp}})$. 
Then $(J,  \langle \cdot,\cdot \rangle)$ is SKT if and only if one of the following condition holds\\
\begin{enumerate} [label=\roman{*}.,ref=(\roman{*}), resume] 
\item[(i)]  $A:=ad_{U} \vert_{\frak h},B:=ad_{JU} \vert_{\h} \in \mathfrak{so}(\h)$ and $[A,J_\h]=[B,J_\h]=0$; \\
\item[(ii)] there exists an orthonormal basis $\{e_1,\dots,e_{2n-2}\}$ of $(\h, \langle \cdot, \cdot \rangle_{\h})$ with respect to which 
\[
J_\h=\operatorname{diag}(\Lambda_1,\dots,\Lambda_{n-1}), \ A=\operatorname{diag}(A_1,\dots,A_{n-1}), \ B=\operatorname{diag}(B_1,\dots,B_{n-1}),
\]
where 
\begin{equation} \label{normalforms}
\Lambda_i=\begin{pmatrix}
0 & -1 \\
1 & 0 \\
\end{pmatrix}, \ 
A_i=\begin{pmatrix}
0 & a_i \\
-a_i & 0 \\
\end{pmatrix} \ \text{and} \  B_i=\begin{pmatrix}
0 & b_i \\
-b_i & 0 \\
\end{pmatrix},
\end{equation}
with $a_i,b_i \in \R$.
\end{enumerate}
Moreover, if  the Hermitian structure $(J,  \langle \cdot,\cdot \rangle)$ is SKT, then \\
\begin{enumerate} [label=\roman{*}.,ref=(\roman{*}), resume] 
\item[(iii)]  $(J, \langle \cdot,\cdot \rangle)$ is K\"ahler if and only if $[\h^\perp,\h^\perp]=0$;\\
\item[(iv)]  if $\mathfrak{z} (\g) \cap \h =\{0\},$ then $\g$ admits also a K\"ahler metric.
\end{enumerate}
\end{corollary}
\begin{proof} 
One direction of \emph{(i)} follows by Theorem \ref{theorem:sktinvariant}.  Indeed,  if $( \langle \cdot,\cdot \rangle, J)$ is SKT and  $\h$ is abelian, then $A, B \in \mathfrak{so}(\h)$ and $[A,J_\h]=[B,J_{\h}]=0$. Let us prove the converse. We have already observed that $\omega=\omega_\h + u\wedge Ju$. Then, 
\[
d\omega=u \wedge A^*\omega_\h + Ju \wedge B^*\omega_\h - u\wedge  Ju \wedge \iota_V \omega_\h,
\] 
with  $V = [U, JU]$, $A^*\omega_\h (Y,W)=  -\langle (J_\h A+ A^tJ_\h )Y, W \rangle_\h$, and $B^*\omega_\h (Y,W)=  -\langle (J_\h B+ B^tJ_\h )Y, W \rangle_\h$ for any $Y,Z \in \h$. Since $[A,J_\h]=[B,J_\h]=0$ and $A,B \in \mathfrak{so}(\h)$
\[
J_\h A+ A^tJ_\h=(A+A^t)J_\h=0, \ J_\h B+ B^tJ_\h=(B+B^t)J_\h=0,
\]
from which follows that $d\omega=- u\wedge  Ju \wedge \iota_V \omega_\h$. The Bismut torsion $3$-form $c$  is hence given by 
\[
c=Jd\omega=-u\wedge  Ju \wedge J(\iota_V \omega_\h)= u\wedge  Ju \wedge \iota_{JV} \omega_\h,
\]
and it is clearly closed as $d\h^* \subset \h^* \otimes x \oplus \h^*\otimes Jx$. \\
 To prove \emph{(ii)} we use that by  \emph{(i)} the SKT condition is equivalent to $[A,J_\h]=[B,J_\h]=0$ and $A,B\in \mathfrak{so}(\h)$. Hence, since  $\h$ is abelian, $A,B,J_\h$ are three skew-symmetric endomorphisms which commute pairwise. As a consequence, there   exists an orthonormal basis $\{e_1,\dots,e_{2n-2}\}$ of $(\h, \langle \cdot,\cdot \rangle_{|\h})$ such that $A,B,J_\h$ can be put simultaneously in their diagonal normal forms
\[
J_\h=\operatorname{diag}(\Lambda_1,\dots,\Lambda_{n-1}), \ A=\operatorname{diag}(A_1,\dots,A_{n-1}), \ B=\operatorname{diag}(B_1,\dots,B_{n-1}),
\]
where  $\Lambda_i$, $A_i$ and $B_i$  are gives as in \eqref{normalforms}.\\
\emph{(iii)}  follows from the fact that,   if $(J, \langle \cdot,\cdot \rangle)$ is SKT then $d\omega=- u\wedge  Ju \wedge \iota_V \omega_\h$. \\
To prove \emph{(iv)}  we exploit the condition $\mathfrak{z} (\g) \cap \h =\{0\}$ to prove that $B-AJ_\h$ is invertible. Indeed, since $(J,  \langle \cdot,\cdot \rangle, J)$ is SKT, by  \emph{(ii)}  we may always find an orthonormal basis $\{e_1,\dots,e_{2n-2}\}$ of $(\h, \langle \cdot,\cdot \rangle_{|\h})$ such that $A,B,J_\h$ are in their diagonal normal forms. Then, 
\[
B-AJ_\h=\operatorname{diag}(C_1,\dots,C_{n-1}), \quad 
{\text{with }}
\quad 
C_i=\begin{pmatrix}
a_i & -b_i \\
b_i & a_i\\
\end{pmatrix}.
\]
In particular,
$
\det(B-AJ_\h)=\prod \det(C_i)=\prod(a_i^2+b_i^2) \neq 0.
$
In fact, if $\det(B-AJ_\h)=0$, then there must exists an index $\overline{i}$ such that $a_{\overline{i}}^2+b_{\overline{i}}^2=0$. Moreover, this would imply that $a_{\overline{i}}=b_{\overline{i}}=0$, i.e., $e_{2\overline{i}-1},e_{2\overline{i}} \in \mathfrak{z} (\g)\cap \h=\{0\}$, a contradiction. \\  
Since  $B-AJ_\h$ is invertible, there exists a vector $Y \in \h$ such that $(B-AJ_\h) Y=V$. We consider the new $J$-Hermitian metric 
$
\langle \cdot,\cdot \rangle'= \langle \cdot,\cdot \rangle_{|\h}+u'^2+Ju'^2,
$
where $u'$ and $Ju'$ are the duals of $U'=U+Y$ and $JU'=JU+JY$. Then, $A'=ad_{u'|\h}=A$, \ $B'=ad_{Ju'|\h}=B$, implying that the Hermitian structure $(J, \langle \cdot,\cdot \rangle')$ is again SKT (observe that $\langle \cdot,\cdot \rangle'_{|\h}= \langle \cdot,\cdot \rangle_{|\h}$). Moreover, since $[U',JU']=0$, the Hermitian structure $(\langle \cdot,\cdot \rangle', J)$ is K\"ahler by \emph{(iii)}. 
\end{proof}
\begin{remark}Note that   by Theorem \ref{theorem:sktinvariant} the Hermitian Lie algebra $(\g, \langle \cdot, \cdot \rangle, J)$ is always Chern Ricci-flat, but in  general, when $\h$ is abelian,  it is not Chern flat (for further details see \cite[Proposition 4]{GZ}).
\end{remark}
Regarding  the existence of generalized K\"ahler structures we can prove the following
\begin{theorem}  \label{thm:equivalent}
Let $\g$ be a solvable Lie algebra with nilradical $\h$ of codimension $2$. 
The following are equivalent
\begin{enumerate} [label=\roman{*}., ref=(\roman{*})]
\item[(i)]  $\g$ admits a generalized K\"ahler structure $(J_\pm, \langle \cdot,\cdot \rangle)$  such that $J_\pm \h=\h$; 
\item[(ii)] $\g$ admits a K\"ahler structure $(J_+, \langle \cdot,\cdot \rangle)$ such that $J_+ \h=\h$. \end{enumerate}
\end{theorem}
\begin{proof} \hfill\\
$(ii) \Rightarrow (i).$ It suffices to take $J_-=-J_+$.\\
$(i) \Rightarrow (ii). $ Let us fix an orthonormal basis $\{U,U'\}$ of $\h^\perp$ with dual basis $\{u,u'\}$. Without loss of generality, we may assume such that $U'=J_+U$  and $J_-U=\varepsilon U'$, for $\varepsilon \in \{-1,+1\}$. Let $A=ad_{U}\vert_{ \h},B=ad_{U'}\vert_{ \h}$ and $V=[U,U']$.   \\
Since the Bismut torsions  $3$-forms $c_\pm$ of the Hermitian structures $(J_\pm, \langle \cdot,\cdot \rangle)$  satisfy $c_+=-c_-$, then there exist $\alpha, \beta \in {\bigwedge}^2 \h^*$ and $\gamma \in \h^*$ such that
\[
  c_\pm= \pm \alpha \wedge u \pm \beta \wedge u' \pm \gamma \wedge u\wedge u' +c_{\h_{\pm}},
\]
where $c_{\h_{\pm}}$ are the torsion $3$-forms of the Hermitian structures $({J_{\pm}}_{\h}, \langle \cdot,\cdot \rangle_\h)$ on $\h$, respectively. Since $c_+ = - c_-$,  we clearly have $c_{\h+}=-c_{\h-}$. \\
Using \eqref{eqn:differential} and the SKT condition $dc_+=0$, we get 
\[
\begin{split}
dc_+&=(d_\h\ \alpha-A^*c_{\h+})\wedge u+(d_\h\ \beta-B^*c_{\h+})u' \\
&\quad +( -B^*\alpha+A^*\beta+d_\h \gamma-\iota_V c_{\h+})\wedge u \wedge u' +d_\h c_{\h+}=0. 
\end{split}
\] 
It follows that $d_\h c_{\h+}=0$, i.e.,  $(\h, {J_{\pm}}_{\h}, \langle \cdot,\cdot \rangle_\h)$ is a generalized K\"ahler Lie algebra. Moreover, since $\h$ is nilpotent, this forces $\h$ to be abelian by \cite{CA}.\\
For any $W \in \h$
\[
c_+ (W,U,U')=-\langle [J_+W,J_+U],U' \rangle - \langle [J_+U,J_+U'],W \rangle -\langle [J_+ U',J_+ W], U\rangle
= - \langle V,W \rangle
\] 
and, analogously, $
c_- (W,U,U')= -\varepsilon^2 \langle V ,W \rangle= - \langle V,W \rangle.
$
Since $c_+=-c_-$, we must have $V=0$, i.e., $[\h^\perp, \h^\perp]=0$. Moreover, this implies that the SKT structure $(J_+, \langle \cdot,\cdot \rangle)$ on $\g$ is K\"ahler by Thorem \ref{thm:abelian}.
\end{proof}
\begin{remark} One can show that if a solvable Lie algebra $\g$ with nilradical $\h$ of even-codimension admits a generalized K\"ahler structure $(\langle \cdot,\cdot \rangle, J_\pm)$ such that $J_\pm \h=\h$, then $\h$ is abelian. The proof proceeds in the same way as before, using a generalization of \eqref{eqn:differential}.
\end{remark}

\section{Case $J\h \neq \h$}\label{section3}
Let $\g$ be a $2n$-dimensional solvable Lie algebra with nilradical $\h$ of codimension $2$ endowed with  an almost Hermitian structure  $(J, \langle \cdot,\cdot \rangle)$  such that $J\h \neq \h$.  As a consequence $\g=\h + J\h$. Setting $\h_J:= \h \cap J\h$, we  have the  orthogonal decomposition
$
\g = \h_J \oplus (\h_J)^\perp, 
$
where each summand is $J$-invariant. Since $\h$ has codimension $2$, $\dim(\h_J)=2n-4$. Observe that when $2n=4$, the decomposition above is trivial. Moreover, since SKT structures on $4$-dimensional solvable Lie algebras have been fully descripted in \cite{MS}, we may restrict to consider $2n > 4$.  Now, we focus on the $4$-dimensional $J$-invariant space $(\h_J)^\perp=\mathfrak{k} \oplus \h^\perp$, where $\mathfrak{k}$ is the orthogonal complement of $\h_J$ in $\h$. 
Let $\{e_{2n-1},e_{2n}\}$ be any orthonormal basis of $\h^\perp$. Then, 
$$
J e_{2n-1}= J_{34}  e_{2n} + h_{2n-2}, \quad 
J e_{2n}= -J_{34}  e_{2n-1} + h_{2n-3}
$$
where $J_{34} \in \R$ and $h_{2n-2}, h_{2n-3}$  is a pair of non-zero orthogonal vectors of $\mathfrak{k}$ such that $J^2 e_{2n-1}=-e_{2n-1}$ and $J^2 e_{2n}=-e_{2n}$. 
Hence, $\{e_{2n-3}:= \frac{h_{2n-3}}{\norm{h_{2n-3}}},e_{2n-2}:=\frac{h_{2n-2}}{\norm{h_{2n-2}}},e_{2n-1},e_{2n}\}$ is an orthonormal basis of the $J$-invariant subspace $ (\h_J)^\perp$. \\
With respect to the decomposition $\g=\h_J \oplus (\h_J)^\perp$, the almost complex structure $J$ splits as 
\[ 
J:= \begin{pmatrix}
     J_{\h_{J}} & 0 \\
    0  &  J_{(\h_J)^\perp}
     \end{pmatrix},
\]
and $J^2=-Id$ is equivalent to $J_{\h_{J}}^2=-Id_{2n-4}$ and $J_{(\h_J)^\perp}^2=-Id_4$. 
With respect to the orthonormal basis $\{e_{2n-3},e_{2n-2},e_{2n-1}, e_{2n}\}$ of $(\h_J)^\perp$, the restricted almost complex structure $J_{(\h_J)^\perp}$ is represented by the skew-symmetric matrix
\begin{equation*}  
J_{(\h_J)^\perp}= \begin{pmatrix} 
0 & J_{12} & 0 & J_{14} \\
-J_{12} & 0 & J_{23} & 0 \\
0 & -J_{23} & 0 & J_{34} \\
-J_{14} & 0 & -J_{34} & 0 \\
\end{pmatrix},
\end{equation*}
where the entries  $J_{ij}$  satisfy the conditions
\begin{equation} \label{eqn:Isquare} 
\begin{cases}
J_{12}^2+J_{14}^2=1, \quad 
J_{12}^2+J_{23}^2=1, \quad 
J_{23}^2+J_{34}^2=1, \quad 
J_{14}^2+J_{34}^2=1, \\
-J_{14}J_{34}+J_{12}J_{23}=0,  \quad 
-J_{14}J_{12}+J_{23}J_{34}=0.
\end{cases}
\end{equation}
Observe that $J_{14} \neq 0$ and $J_{23} \neq 0$, 
as $J\h\neq \h$. 
Exploiting the conditions \eqref{eqn:Isquare}, one obtains two different but equivalent almost complex structures on $(\h_J)^\perp$ corresponding to either $J_{12}=-J_{34}$ and $J_{14}=-J_{23}$ or $J_{12}=J_{34}$ and $J_{14}=J_{23}$. We will consider only the first case, as they are equivalent up to a change of basis of $(\h_J)^\perp$.
\begin{remark} \label{rem:cmpstr}
To summarize, we may always endow $(\h_J)^\perp$ with an orthonormal basis $\{e_{2n-3},e_{2n-2},e_{2n-1}, e_{2n}\}$ such that $\{e_{2n-3}, e_{2n-2}\}$ and $\{e_{2n-1}, e_{2n}\}$ are unitary basis of $\mathfrak{k}$ and $\h^\perp$ respectively, and the restricted almost complex structure $J_{ (\h_J)^\perp}$ can be written with respect to such a basis as
\begin{equation} \label{eqn:complexstructure2} 
J_{ (\h_J)^\perp}= \begin{pmatrix}
0 & J_{12} & 0 & J_{14} \\
-J_{12} & 0 & -J_{14} & 0 \\
0 & J_{14} & 0 & -J_{12} \\
-J_{14} & 0 & J_{12} & 0 \\
\end{pmatrix},
\end{equation}
with $J_{12}^2+J_{14}^2=1$ and $J_{14} \neq 0$. 
\end{remark}
We restrict  now to the case of $\h$ being abelian, where we can investigate the existence of generalized K\"ahler structures. 
\begin{proposition} \label{proposition:integrability}
Let  $(\g,  J, \langle \cdot, \cdot \rangle)$ be an Hermitian solvable Lie algebra with codimension $2$ abelian nilradical $\frak h$ such that $J\h \neq \h$. Then $\h_J:=\h \cap J\h$ is an ideal of $\g$. \\
Moreover,  if $\{e_{2n-3}, e_{2n-2}, e_{2n-1}, e_{2n}\}$ is an orthonormal basis of $(\h_J)^\perp$ as in Remark \ref{rem:cmpstr}, we further have that $$ [A_{\h_{J}},J_{\h_{J}}] =[B_{\h_{J}},J_{\h_{J}}]=0,$$ where $A_{\h_{J}}, B_{\h_{J}},J_{\h_{J}}$ denote the restriction of $A:=ad_{e_{2n-1}}\vert_{\h}, B:=ad_{e_{2n}}\vert_{\h}$ and $J$ to $\h_J$, respectively.
\end{proposition}
\begin{proof}
Let $\{e_{2n-3},e_{2n-2},e_{2n-1},e_{2n}\}$ be an orthonormal basis of $(\h_J)^\perp$ as in Remark \ref{rem:cmpstr}. To prove that $\h_J$ is an ideal, since $\h$ is abelian and $\g^1 \subset \h$, we only need to check that for any $X \in \h_J$, $[X, e_{2n}],[X, e_{2n-1}] \in J\h$. By the integrability of $J$
\[
\begin{split}
N_J(X,e_{2n-2})=&J[X,Je_{2n-2}]-[JX,J e_{2n-2}]\\
=& J[X,J_{12}e_{2n-3}+J_{14} e_{2n-1}]-[JX,J_{12}e_{2n-3}+J_{14} e_{2n-1}] \\
=& J_{14}  (J[X,e_{2n-1}]-[JX,e_{2n-1}] )= -J_{14} (JAX-AJX)= 0.
\end{split}
\]
Hence, $[X,e_{2n-1}]=J[-JX,e_{2n-1}] \in J\h$. Analogously, $[X,e_{2n}]=J[-JX,e_{2n}] \in J\h$. Furthermore, we also get $ [A_{\h_{J}},J_{\h_{J}}] =[B_{\h_{J}},J_{\h_{J}}]=0$.
\end{proof}

\begin{proposition} \label{proposition:gk}
Let $\g$ be a solvable Lie algebra with abelian nilradical $\h$ of codimension $2$. Assume that  $\g$ is endowed with a SKT structure $(J, \langle \cdot, \cdot \rangle)$ such that $J\h= \h$. If there exists another complex structure $I$ compatible with $\langle \cdot,\cdot \rangle$ and such that $I\h \neq \h$, then $( J, \langle \cdot ,\cdot \rangle)$ is K\"ahler. \\
In particular $\g$ does not admit any non K\"ahler generalized K\"ahler structure $(I,  J, \langle \cdot,\cdot \rangle)$ such that  $I \h \neq \h$ and $J\h=\h$.
\end{proposition}
\begin{proof}
Let $\{e_{2n-3},e_{2n-2},e_{2n-1},e_{2n}\}$ be an orthonormal basis of $(\h_I)^\perp=(\h \cap I\h)^\perp$ as in Remark \ref{rem:cmpstr}). \\
If we consider any orthonormal basis $\{e_1,\dots,e_{2n-4}\}$ of $\h_I$, then  $\cal{B}=\{e_1,\dots,e_{2n-4},e_{2n-3},e_{2n-2}\}$ is an orthonormal basis of $\h$. Furthermore, $\h_I$ is an ideal of $\g$ by Proposition \ref{proposition:integrability}. \\
Since $(\langle \cdot, \cdot \rangle, J)$ is SKT and $\h$ is abelian, we have that  $A= ad_{e_{2n-1}}\vert_{\h}, B = ad_{e_{2n}}\vert_{\h}$ are in $\mathfrak{so}(\h)$ by Corollary \ref{thm:abelian}. Hence, with respect to $\cal{B}$ 
\[
A= \begin{pmatrix}
\begin{array}{c|c}
A_{\h_{I}}& 0\\
\hline
0 & \begin{matrix} 0 & c_{12} \\
-c_{12} & 0 \end{matrix} 
\end{array}
\end{pmatrix}, \ \ 
B= \begin{pmatrix}
\begin{array}{c|c}
B_{\h_{I}}& 0\\
\hline
0 & \begin{matrix} 0 & d_{12} \\
 -d_{12} & 0 \end{matrix} 
\end{array}
\end{pmatrix},
\quad 
{\text {where}}  \, \, A_{\h_{I}}, B_{\h_{I}} \in \mathfrak{so}(\h_I). 
\]
Regarding the Nijenhuis tensor $N_I$ as a $(0,3)$-tensor with the aid of the inner product $\langle \cdot,\cdot \rangle$, namely $N_I(X,Y,Z)=\langle N_I(X,Y),Z \rangle$, we get
\[
N_I(e_{2n-2},e_{2n-3},e_{2n-1})= -I_{14} c_{12}, \ \  N_I(e_{2n-2},e_{2n-3},e_{2n})= -I_{14} d_{12}.
\]
By the integrability of $I$ and by the fact that $I_{14} \neq 0$ (see Remark \ref{rem:cmpstr}), we have $c_{12}=d_{12}=0$. \\
Hence, if one computes
\[
[e_{2n-1},e_{2n}]=-(I[Ie_{2n-1},e_{2n}]+I[e_{2n-1},Ie_{2n}])+[Ie_{2n-1},Ie_{2n}]= I_{12}^2 [e_{2n-1},e_{2n}],
\]
one gets $0=(1-I_{12}^2)[e_{2n-1},e_{2n}]=I_{14}^2 [e_{2n-1},e_{2n}]$, i.e., $[\h^\perp,\h^\perp]=0$. The result follows by Corollary \ref{thm:abelian}. 
\end{proof}
\begin{remark}
We have seen in the previous section that if  $\g$ is a solvable Lie algebra with codimension $2$ nilradical $\h$ that admits a generalized K\"ahler structure $(I_{\pm}, \langle \cdot,\cdot \rangle)$ such that $I_\pm \h=\h$, then $\h$ is abelian. It is in general not true when  $I_\pm \h \neq\h$. Indeed, let us consider the Lie algebra $\frak s \cong (\mathfrak{h}_3\oplus \R^3) \rtimes \R^2 =(e^{23}+e^{17},\frac{1}{2} e^{27},\frac{1}{2} e^{37},-e^{48},\frac{1}{2} e^{58} -e^{67},\frac{1}{2} e^{68}+e^{57},0,0),$ where $\{e_1,\dots,e_{6}\}$ is a basis of $\mathfrak{h}_3\oplus \R^3$ , $\{e_7,e_8\}$ is a basis of $\R^2$ and by $e^{ij}$ we denote $e^i \wedge e^j$.
By the structure equations, we have that the nilradical $\h$ is spanned by $\{e_1,\dots,e_{6}\}$ and it is easy to observe that $\h=\mathfrak{h}_3\oplus \R^3$. Let us define the bi-Hermitian structure $(\langle \cdot,\cdot \rangle, I_\pm)$ with  $I_\pm \h\neq \h$ as
\[
 I_\pm e_1=e_{7}, \  I_\pm e_2= e_{3},\  I_\pm e_{5}=\pm e_{6},  I_\pm e_{4}=e_{8}, \quad  \langle \cdot,\cdot \rangle =  \mathlarger{\sum}_{i=1}^{8} (e^i)^2. \ 
\]
The corresponding fundamental forms $\omega_\pm=e^1 \wedge e^{7} + e^2 \wedge e^3 \pm e^{5}\wedge e^{6}+e^4 \wedge e^8$ satisfy
\[
\begin{split}
d^c_\pm\omega_\pm= \pm \ e^4 \wedge e^5 \wedge e^6,
\end{split}
\]
and  $e^{4}\wedge e^{5}\wedge e^{6}$ is a closed $3$-form, i.e., $( I_\pm, \langle \cdot,\cdot \rangle)$ is a generalized K\"ahler structure. Moreover, we observe that since the  Lie algebra is not unimodular, the corresponding connected and simply connected Lie group does not admit lattices.  
\end{remark}
\section{Classification in dimension six} \label{section4}
In this section we provide a classification of six-dimensional unimodular solvable Lie algebras with nilradical of codimension $2$ that admit a SKT structure.  
\begin{theorem} \label{thm:classif} A unimodular six-dimensional solvable Lie algebra  $\g$ with codimension $2$ nilradical $\h$ admits a SKT structure $(J, \langle \cdot,\cdot \rangle)$ if and only if $\g$ is isomorphic to one of the following Lie algebras
\begin{flalign*} 
&\begin{aligned}
 &\tau_{3,0} \times \tau_{3,0}=(-f^{25},f^{15},-f^{46},f^{36},0,0),\\
 &\g_{5.35}^{-2,0} \oplus \R=(2f^{15},-f^{25}-f^{36},-f^{35}+f^{26},0,0,0),
\end{aligned} && 
\end{flalign*} hence, in particular, $\h$ must be abelian. An explicit example of  SKT structure is given respectively by 
$$
\begin{aligned}
&Jf_1=f_2, \ Jf_3=f_4, \ Jf_5=f_6, \  \langle \cdot,\cdot \rangle=\sum_{i=1}^6 (f^i)^2\\
&Jf_1=f_5, \ Jf_2=f_3, \ Jf_4=f_6, \  \langle \cdot,\cdot \rangle=\sum_{i=1}^6 (f^i)^2.
\end{aligned}$$
\end{theorem}
\begin{proof}
We discuss separately the cases $J\h =\h$ and $J\h \neq \h$.\\
Assume that $\g$ is endowed with a SKT structure $(J, \langle \cdot,\cdot \rangle)$ such that $J\h =\h$. By Theorem \ref{theorem:sktinvariant}, $\h$ is a $4$-dimensional nilpotent SKT Lie algebra, so we have that either $\h= \R^4$ or $\h = \mathfrak{h}_3 \oplus \R$.  If $\h$ is abelian, then one can easily prove that $\mathfrak{z}(\g) \cap \h = 0$ (otherwise we get a contradiction with the fact that $\h$ has codimension $2$) and so by the proof of statement \emph{(iv)} in Corollary \ref{thm:abelian}, we get that $\h=[\g,\g]$.  The classification follows by \cite{FS,FS2} and the Lie algebra $\g$ has to be isomorphic to $\tau_{3,0} \times \tau_{3,0}$.\\
%We begin by considering the case of $\h$ being abelian, in order to exploit the results given in Corollary \ref{thm:abelian}. Observe that in this case the unimodular assumption is unnecessary, since it is already implied by the SKT condition.  \\
%By Corollary \ref{thm:abelian}, we have that there must exist an orthonormal basis of $\h$ such that 
%\[ 
%A=ad_{e_{5}}\vert_{\h}=\begin{pmatrix}
%\begin{array}{cc|cc}
%0 & -a_1 \\
%a_1 & 0 \\ \hline
% & & 0 & -a_2 \\
% & & a_2 & 0
% \end{array}
% \end{pmatrix}, \ 
% B=ad_{e_{6}}\vert_{\h}=\begin{pmatrix}
% \begin{array}{cc|cc}
%0 & -b_1 \\
%b_1 & 0 \\ \hline
% & & 0 & -b_2 \\
% & & b_2 & 0
% \end{array}
% \end{pmatrix},
%\]
%where $\{e_5,e_6\}$ is an orthonormal basis of $\h^\perp$ such that $J e_5=e_6$.
%One can easily prove that  $\mathfrak{z}(\g) \cap \h = 0$ (otherwise we get a contradiction with the fact that $\h$ has codimension $2$), and so by the proof of statement \emph{(iv)} in Corollary \ref{thm:abelian} we get that $\g \cong \tau_{3,0} \times \tau_{3,0}$.  \\
Now we deal to the case $\h = \mathfrak{h}_3 \oplus \R$, proving that this case cannot occur. 
Assume by contradiction that $\g$ admits a SKT structure $(\langle \cdot, \cdot \rangle, J)$ such that $J \h =\h$, then 
\[
\g = \h \oplus \h^\perp= \mathfrak{z}(\h) \oplus \mathfrak{z}(\h) \oplus \h^\perp.
\]
Let $e_1$ be a generator of $\h^1=[\h,\h]\subset \mathfrak{z}(\h)$ which, up to rescaling, we may assume to be unitary. Then an orthonormal basis of $\g$ is provided by $\cal{B}=\{e_{1},e_{2}=Je_{1},e_3,e_4=Je_{3},e_{5},e_{6}=Je_{5}\}$, where $\{e_1,e_{2}\}$ is a basis of $ \mathfrak{z}(\h)$, $\{e_1,e_{2}\}$ is a basis of $  \mathfrak{z}(\h)^{\perp}$ and $\{e_5,e_6\}$ is a basis of $\h^\perp$. \\
The Lie algebra $\g$ is completely determined by the data 
\[
[e_5,X]=AX, \ [e_6,X]=BX, \ \forall X \in \h\ [e_3,e_4]=\eta e_1, \ \ [e_5,e_6]=V,
\]
where $\eta \in \R \setminus \{0\}$, $A=ad_{e_5}\vert_{\h}, B=ad_{e_6}\vert_{\h}$ are derivations of $\h$ satisfying $[A,B]=ad_{V}\vert_{\h}$.
We have already observed in the first section that we may decompose $A$ and $B$ as 
\[
A=
\begin{pmatrix}
A_{\mathfrak{z}(\h)} & \ast_A \\
0 & A_{\mathfrak{z}(\h)^\perp}
\end{pmatrix} \ \ \text{and} \ \ 
B=\begin{pmatrix}
B_{\mathfrak{z}(\n)} & \ast_B \\
0 & B_{\mathfrak{z}(\h)^\perp}
\end{pmatrix},
\] 
with $A_{\mathfrak{z}(\h)}, B_{\mathfrak{z}(\h)} \in \mathfrak{so}(\mathfrak{z}(\h))$ being such that $[A_{\mathfrak{z}(\h)} ,J_{\mathfrak{z}(\h)} ]=[B_{\mathfrak{z}(\h)} ,J_{\mathfrak{z}(\h)} ]=0$ by Theorem \ref{theorem:sktinvariant}. \\
Moreover, since $A$ and $B$ are derivations of $\h$, $A \h^1 \subset \h^1$ and $  B \h^1 \subset \h^1$, and so $A_{\mathfrak{z}(\h)}=B_{\mathfrak{z}(\h)}=0$.  In particular, this forces $\tr(A_{\mathfrak{z}(\h)^\perp})=\tr(B_{\mathfrak{z}(\h)^\perp})=0$.\\
Let 
\[
A_{\mathfrak{z}(\h)^\perp}= \begin{pmatrix} 
a_{33} & a_{34} \\
a_{43} & -a_{33}
\end{pmatrix} \ \ \text{and} \ \
B_{\mathfrak{z}(\h)^\perp}= \begin{pmatrix} 
b_{33} & b_{34} \\
b_{43} & -b_{33}
\end{pmatrix}.
\]
Exploiting that $A_{\mathfrak{z}(\h)}=B_{\mathfrak{z}(\h)}=0$, the integrability condition $[J_\h,A]J_\h+[J_\h,B]=0$ and the relation $[A,B]=ad_{V}\vert_{ \h}$ yield
\[
\begin{split}
&[J_\h,A_{\mathfrak{z}(\h)^\perp}]J_\h+[J_\h,B_{\mathfrak{z}(\h)^\perp}]=0 \iff 2a_{33}=b_{34}+b_{43}, \ \ 2b_{33}=-(a_{34}+a_{43}),\\
&[A_{\mathfrak{z}(\h)^\perp},B_{\mathfrak{z}(\h)^\perp}]=0 \iff 2a_{33}(b_{34}+b_{43}) - 2b_{33}(a_{34}+a_{43})=0. \\
\end{split}
\]
It is immediate to observe that the two conditions together imply that $A_{\mathfrak{z}(\h)^\perp}$ and $B_{\mathfrak{z}(\h)^\perp}$ are skew. 
If any of $a_{34}$ or $b_{34}$ is zero, then the corresponding matrix would be nilpotent, which is a contradiction with the hypothesis that $\h$ has codimension $2$. Hence we must have $a_{34}, b_{34} \neq 0$. Moreover, if we consider the subspace generated by $\{e_1,\dots,e_{4},e_5-\frac{a_{34}}{b_{34}}e_6\}$, then it is a nilpotent ideal which strictly contains $\h$, providing the contradiction. \\
Now, we consider the case of $J\h \neq \h$. In this case, the four-dimensional nilradical is isomorphic to one of $\R^4, \mathfrak{h}_3 \oplus \R$ and $ \h_4=(-24,-34,0,0)$. Moreover, the only six-dimensional solvable Lie algebra with nilradical $\h_4$ is $\g_{6.28}=(0,0,46-13-2.25,56-24,15,-16+26)$ (see \cite{RT}), which is not unimodular. \\
Hence, as in the previous case, we may restrict to consider either $\h = \mathfrak{h}_3 \oplus \R$ or $\h = \R^4$. We start by considering the first case, which is more involved. \\
Firstly, we prove that if $\g$ admits a complex structure $J$ such that $J\h\neq \h$,  then $\dim(J (\mathfrak{z}(\h)) \cap \h)=1$. Since the center $\mathfrak{z}(\h)$ has dimension $2$, it suffices to check that if $\dim(J (\mathfrak{z}(\h)) \cap \h)=0,2$ we get a contradiction. \\
We start by considering $\dim(J (\mathfrak{z}(\h)) \cap \h)=0$. Let $Z_1$ be a generator of $\h^1=[\h,\h]$, and let $Z_2 $ be such that $\mathfrak{z}(\h)=\text{span}_\R \langle Z_1,Z_2 \rangle$. Since $\g=\h + J\h$, we may fix a basis $\{Z_1,Z_2,X,JX,JZ_1,JZ_2\}$ of $\g$, with $X \in \h_J=\h \cap J\h$. \\
Observe that since $ad_{JZ_{1}}\vert_{\h},ad_{JZ_{2}}\vert_{\h}$ are derivations of $\h$, they preserve both $\h^1$ and $\mathfrak{z}(\h)$. Exploiting this fact, the integrability of $J$ and the inclusion $\g^1 \subset \h$, we get the following structure equations
\[
\begin{split}
&[X,JX]=\eta Z_1,\ \ [Z_1,JZ_1]=a_{11} Z_1,  \ \
[Z_1,JZ_2]=a_{12} Z_1, \ \ [Z_2,JZ_1]=a_{12} Z_1, \\
&[Z_2,JZ_2]=b_{12} Z_1 +b_{22} Z_2, \ \ [X,JZ_1]= a_{33} X+a_{43} JX, \ \
[X,JZ_2]= b_{33} X+b_{43} JX,\\ 
&[JX,JZ_1]= -a_{43} X+a_{33} JX,\ \
[JX,JZ_2]= -b_{43} X+b_{33} JX, \ \  [JZ_1,JZ_2]=0. \\
\end{split}
\]
Computing $$ad_{JZ_1}[X,JX]=[ad_{JZ_1} X,JX]+[X, ad_{JZ_1} JX] $$ and $$ ad_{JZ_2}[X,JX]=[ad_{JZ_2} X,JX]+[X, ad_{JZ_2} JX]$$ we further have $a_{11}=2a_{33}$ and $a_{12}=2b_{33}$. \\
In particular, by the unimodularity condition, this leads to $a_{11}=a_{33}=0$ and $b_{22}=-4b_{33}$. \\
Moreover, $[ad_{JZ_{1}},ad_{JZ_{2}}]=ad_{[JZ_1,JZ_2]}=0$ if and only of  $b_{33}=0$. 
\\
To sum up, the Lie algebra $\g$ is determined by the data 
\begin{equation*}
\begin{split}
ad_{JZ_{1}}\vert_{\h}=\begin{pmatrix}
0 & 0 & 0 & 0 \\
0 & 0 & 0 & 0 \\
0 & 0 & 0 & a_{43} \\
0 & 0 & -a_{43} & 0\\
\end{pmatrix},  \ \ ad_{JZ_{2}}\vert_{\h}=\begin{pmatrix}
0 & -b_{12} & 0 & 0 \\
0 & 0 & 0 & 0 \\
0 & 0 & 0 & b_{43} \\
0 & 0 & -b_{43} &0 \\
\end{pmatrix}, \  \ [X,JX]=\eta Z_1 \neq 0.
\end{split}
\end{equation*}
Observe that we must have $a_{43} \neq 0$. As $a_{43} \neq 0$, we may consider the subspace generated by $\{Z_1,Z_2,X,JX,JZ_2-\frac{b_{43}}{a_{43}}JZ_1 \}$, that is a nilpotent ideal which strictly contains $\h$. The contradiction follows. \\
Let us now consider the case of $\dim(J (\mathfrak{z}(\h)) \cap \h)=2$. In this case $\h_J=\mathfrak{z}(\h)$, so we may fix a basis $\{Z, JZ\}$ of $\mathfrak{z}(\h)$, with $Z$ being a generator of $\h^1$. Since $\g=\h+ J\h$, we may complete $Z,JZ$ to a basis $\{Z,JZ,X,Y,JX,JY\}$ of $\g$, where $\{Z,JZ,X,Y\}$ is a basis of $\h$. Analogously to the previous case, we obtain the following structure equations 
\[
\begin{split}
&[X,Y]=\eta Z,\ \ [Z,JX]=a_{11} Z, \ \
[Z,JY]=b_{11} Z, \ \ [JZ,JX]=a_{11} JZ, \\
&[JZ,JY]=b_{11} JZ, \ \ [X,JX]=a_{13}Z+a_{23}JZ+ a_{33} X+a_{43} Y, \\
&[X,JY]= b_{13}Z+b_{23}JZ+ a_{34} X+a_{44} Y, \ \
 [Y,JX]= a_{14}Z+a_{24}JZ+ a_{34} X+a_{44} Y,\\
&[Y,JY]= b_{14}Z+b_{24}JZ+ b_{34} X+b_{44} Y, \ \
  [JX,JY]=(\eta +(a_{24}-b_{23}))Z+(b_{13}-a_{14})JZ. \\
\end{split}
\]
Exploiting $$ad_{JX}[X,Y]=[ad_{JX} X,Y]+[X, ad_{JX} Y] \quad  {\text {and}}  \quad  ad_{JY}[X,Y]=[ad_{JY} X,Y]+[X, ad_{JY} Y]$$ we further have $a_{11}=a_{33}+a_{44}$ and $b_{11}=a_{34}+b_{44}$. This identities combined with the unimodularity of $\g$ lead to $a_{11}=0$, $a_{33}=-a_{44}$ , $b_{11}=0$ and $a_{34}=-b_{44}$. \\
We hence get that
\begin{equation*}
\begin{split}
&ad_{JX}\vert_{\h}=\begin{pmatrix}
0 & 0 & -a_{13} & -a_{14} \\
0 & 0 & -a_{23} & -a_{24} \\
0 & 0 & -a_{33} & -a_{34} \\
0 & 0 & -a_{43} & a_{33}\\
\end{pmatrix},  \ \ ad_{JY}\vert_{\h}=\begin{pmatrix}
0 & 0 & -b_{13} & -b_{14} \\
0 & 0 & -b_{23} & -b_{24} \\
0 & 0 & -a_{34} & -b_{34} \\
0 & 0 & a_{33} & a_{34}\\
\end{pmatrix}, \\ 
&[X,Y]=\eta Z, \ [JX,JY]=(\eta +(a_{24}-b_{23}))Z+(b_{13}-a_{14})JZ \in \mathfrak{z}(\h)=\mathfrak{z}(\g).
\end{split}
\end{equation*}
Since $[JX,JY]\in \mathfrak{z}(\h)=\mathfrak{z}(\g)$, as before $[ad_{JX},ad_{JY}]=ad_{[JX,JY]}=0$. With the relations coming from this identity, if one compute the spectrum of $ad_{JX}$ one obtains that $ad_{JX}$ has only the eigenvalue $0$, so it is nilpotent. Hence, the subspace generated by $\{Z,JZ,X,Y,JX\}$ is a nilpotent ideal which strictly contains $\h$, giving a contradiction. \\
Hence, we may restrict to consider the case of $\dim(J (\mathfrak{z}(\h)) \cap \h)=1$ and we prove that if $\dim(J (\mathfrak{z}(\h)) \cap \h)=1$, then $\g$ cannot admit any SKT structure $(\langle \cdot,\cdot \rangle, J)$ such that $J\h \neq \h$. \\
 Let $Z_1,Z_2$ be a basis of $\mathfrak{z}(\h)$ such that $JZ_1 \in \h$ and $JZ_2 \notin \h$. Since $\g=\h+J\h$, there always exists a vector $X \in \h$ such that $\{Z_1,JZ_1,Z_2,X,JZ_2,JX\}$ is a basis of $\g$. We proceed as before. Indeed, using that $ad_{JZ_2},ad_{JX}$ preserve $\mathfrak{z}(\h)$, the integrability of $J$, and the inclusion $\g^1 \subset \h$, we get 
\[
\begin{split}
&[JZ_1,X]=\eta_1 Z_1+\eta_2 Z_2,\ \ [Z_1,JZ_2]=a_{11} Z_1,  \ \ 
[Z_1,JX]=b_{11} Z_1-\eta_2 Z_2, \\
& [JZ_1,JZ_2]=a_{11} JZ_1,  \ \ 
[JZ_1,JX]=(\eta_1+b_{11}) JZ_1, \ \ [Z_2,JZ_2]= a_{13} Z_1+a_{33} Z_2, \\
&[Z_2,JX]= b_{13} Z_1+a_{34} Z_2, \ \ [X,JZ_2]= a_{14} Z_1+a_{24} JZ_1+a_{34}Z_2,\\
&[X,JX]= b_{14} Z_1+b_{24} JZ_1+b_{34}Z_2+b_{44} X, \ \  [JZ_2,JX]=a_{24}Z_1+(b_{13}-a_{14})JZ_1. \\
\end{split}
\]
By the unimodularity of $\g$, $a_{33}=-2a_{11}$ and $2b_{11}+\eta_{1}+a_{34}+b_{44}=0$. Observe that if $a_{11}=0$ then $ad_{JZ_{1}}$ is a strictly upper triangular matrix, and hence, nilpotent. Hence $a_{11} \neq 0$, and, exploiting $ad_{JZ_2}[JZ_1,X]=[ad_{JZ_2} JZ_1,X]+[JZ_1, ad_{JZ_2} X]$, we get $\eta_2=0$. \\
Analogously, if one computes $ad_{JX}[JZ_1,X]=[ad_{JX} JZ_1,X]+[JZ_1, ad_{JX} X]$, then one obtains $b_{44}=-\eta_1$. Plugging this identity in the unimodularity condition $2b_{11}+\eta_{1}+a_{34}+b_{44}=0$  this also gives $2b_{11}+a_{34}=0$. \\
To sum up
\begin{equation*}
\begin{split}
&ad_{JZ_{2}}\vert_{\h}=\begin{pmatrix}
-a_{11} & 0 & -a_{13} & -a_{14} \\
0 & -a_{11} & 0 & -a_{24} \\
0 & 0 & 2a_{11} &2b_{11} \\
0 & 0 & 0 & 0\\
\end{pmatrix},  \ \ ad_{JX}\vert_{\h}=\begin{pmatrix}
 -b_{11}& 0 & -b_{13} & -b_{14} \\
0 & -(\eta_1+b_{11}) & 0 & -b_{24} \\
0 & 0 & 2b_{11} & -b_{34} \\
0 & 0 & 0 & \eta_1\\
\end{pmatrix}, \\ &[JZ_2,JX]=a_{24}Z_1+(b_{13}-a_{14})JZ_1.
\end{split}
\end{equation*}
Consider any $J$-hermitian inner product $\langle \cdot,\cdot \rangle$ and denote by $c$ the Bismut torsion  $3$-form of $(J, \langle \cdot,\cdot \rangle, J)$. Then, $
dc(Z_1,JZ_1,X,JZ_2)=-2a_{11}\eta_1 \norm{Z_1}^2 \neq 0.$
Indeed, this can be zero only if $\eta_1=0$. Moreover in this case $\h = \R^4$, giving a contradiction. \\
Hence, we have proved that if $\g$ is endowed with a SKT structure $(\langle \cdot,\cdot \rangle, J)$ such that $J\h \neq \h$, then $\h$ must be abelian. 
We restrict to consider $\h= \R^4$.\\
Let us assume that $\g$ admits a SKT structure $(\langle \cdot,\cdot \rangle, J)$ such that $J\h \neq \h$, with $\h =\R^4$. We fix an orthonormal basis $\cal{B}=\{e_1,\dots,e_{6}\}$ of $\g$ such that $\{e_1,e_{2}\}$ is an orthonormal basis of $\h_J$ such that $Je_1=e_2$ and $\{e_{3},\dots,e_{6}\}$ is an orthonormal basis of $(\h_J)^\perp$ as in Remark \ref{rem:cmpstr}, namely, with respect to $\cal{B}$ the complex structure $J$ can be written as 
\[
J= \begin{pmatrix}
0 & 1 & 0 & 0 & 0 & 0\\
-1 & 0 & 0 & 0 & 0 & 0 \\ 
0& 0 & 0 & J_{12} & 0 & J_{14} \\
0 & 0 & -J_{12} & 0 & -J_{14} & 0 \\
0 & 0 & 0 & J_{14} & 0 & -J_{12} \\
0 & 0 &-J_{14} & 0 & J_{12} & 0 \\
\end{pmatrix}.
\]
By Proposition \ref{proposition:integrability}, $\h_J$ is an ideal of $\g$ and $[A_{\h_{J}}, J]=[B_{\h_{J}}, J]=0$, where $A_{\h_{J}}, B_{\h_{J}}$ denote the restrictions of $A:=ad_{e_{5}}\vert_{\h}, B:=ad_{e_{6}}\vert_{\h}$ to $\h_J$. By the integrability of $J$, we have that
\begin{equation*} 
\begin{split}  
&A= \begin{pmatrix}
\begin{array}{c|c}
\begin{matrix}
a_{11} & a_{12}\\
-a_{12} & a_{11}
 \end{matrix} & \begin{matrix} \ast \end{matrix} \\
\hline
 0 & \begin{matrix} c_{11} & c_{12} \\
 c_{21} & c_{22} \end{matrix} 
\end{array}
\end{pmatrix}, \ 
B= \begin{pmatrix}
\begin{array}{c|c}
\begin{matrix}
b_{11} & b_{12}\\
-b_{12} & b_{11} \end{matrix} & \begin{matrix}  \ast \end{matrix} \\
\hline
 0 & \begin{matrix} d_{11} & -c_{11} \\
 d_{21} & -c_{21} \end{matrix} 
\end{array}
\end{pmatrix}, \\
&[e_5,e_6]=V_{\h_{J}}- \frac{J_{12}}{J_{14}}(c_{12}-d_{11}) e_3- \frac{J_{12}}{J_{14}}(c_{22}-d_{21})e_4,
\end{split}
\end{equation*}
with $2a_{11}+c_{11}+c_{22}=0$ and $2b_{11}+d_{11}-c_{21}=0$ by the unimodularity of $\g$. In the following we will denote by 
\[
C:=\begin{pmatrix} c_{11} & c_{12} \\
 c_{21} & c_{22} \end{pmatrix} 
\ \text{and} \ D:= \begin{pmatrix} d_{11} & -c_{11} \\
 d_{21} & -c_{21} \end{pmatrix}.
\]
Since by the Jacobi identity $[A,B]=0$, then also $[C,D]=0$.\\
By the  SKT condition,  we get
\[
\begin{split}
dc(e_1,e_2,e_3,e_6)=&2 J_{14} (-b_{11} c_{21}+a_{11}d_{21}) =0, \\
dc(e_1,e_2,e_3,e_5)=&2 J_{14} (b_{11} c_{22}+a_{11}c_{21})=0, \\
dc(e_1,e_2,e_4,e_5)=&- 2 J_{14} (b_{11} c_{12}+a_{11}c_{11})=0\\
dc(e_1,e_2,e_4,e_6)=&2 J_{14} (b_{11} c_{11}-a_{11}d_{11})=0\\
\end{split}
\]
Since $J_{14} \neq 0$, the coefficients of $A$ and $B$ must obey to 
\begin{equation} \label{eqn:relations}
\begin{split}
&b_{11} c_{21}-a_{11}d_{21}=0,  \ \ 
b_{11} c_{22}+a_{11}c_{21}=0, \ \  
b_{11} c_{12}+a_{11}c_{11}=0, \ \
b_{11} c_{11}-a_{11}d_{11}=0. 
\end{split}
\end{equation}
We distinguish two cases depending on whether $a_{11}$ is zero or not. We will do in details the first case. In fact, the second one is analogous. \\
Let $a_{11}=0$. The conditions \eqref{eqn:relations} becomes
\[
b_{11} c_{21}=0,\ \ b_{11} c_{22}=0, \ \ b_{11} c_{12}=0, \ \ b_{11} c_{11}=0,
\]
and furthermore $c_{11}=-c_{22}$ in order to have $\tr (A)=0$. \\ 
We claim that $b_{11} \neq 0$. Indeed, assume by contradiction that $b_{11}=0$. Then \eqref{eqn:relations} are satisfied and $d_{11}=c_{21}$ by the unimodularity condition. Using $[C,D]=0$, one can show that $C$ and $D$ are nilpotent matrices and so we must have $a_{12},b_{12} \neq 0$. Indeed, if  for instance $a_{12}=0$, then the subspace generated by $\{e_{1},e_{2},e_{3},e_{4},e_{5}\}$ would be a nilpotent ideal which strictly contains $\h$, which is a contradiction with the maximality of the nilradical.\\
 If $a_{12},b_{12} \neq 0$, then we may consider $e'_5=e_5-\frac{a_{12}}{b_{12}}e_6$ with 
\[ 
ad_{e'_{5}}\vert_{\h}=
\begin{pmatrix}
\begin{array}{c|c}
0 &\ast \\
\hline
0  & E
\end{array}
\end{pmatrix},
\]
where $E=C-\frac{a_{12}}{b_{12}} D$. Moreover, since $C,D$ are nilpotent and $[C,D]=0$, so is $E$ and, hence, also $ad_{e'_{5}}\vert_{\h}$. Therefore, if we consider the ideal $\{e_{1},e_{2},e_{3},e_{4},e'_{5}\}$, then again this is a nilpotent ideal which strictly contains the nilradical, proving the claim. \\
Since $b_{11} \neq 0$, by \eqref{eqn:relations} we must have that 
\begin{equation*} 
\begin{split}  
&A= \begin{pmatrix}
\begin{array}{c|c}
\begin{matrix}
0 & a_{12}\\
-a_{12} & 0
 \end{matrix} & \begin{matrix}  \ast \end{matrix} \\
\hline
 0 & \begin{matrix} 0 & 0\\
 0 & 0 \end{matrix} 
\end{array}
\end{pmatrix}, \ \
B= \begin{pmatrix}
\begin{array}{c|c}
\begin{matrix}
b_{11} & b_{12}\\
-b_{12} & b_{11} \end{matrix} & \begin{matrix}  \ast  \end{matrix} \\
\hline
0  & \begin{matrix}  -2b_{11} & 0 \\
 d_{21} & 0 \end{matrix} 
\end{array}
\end{pmatrix}, \\
&  [e_5,e_6]=V_{\h_{J}}-2 \frac{J_{12}}{J_{14}} b_{11} e_3+ \frac{J_{12}}{J_{14}}d_{21}e_4,
\end{split}
\end{equation*}
with $a_{12}\neq 0$ (otherwise $A$ would be nilpotent).
In order to kill the components of $V$ along $e_3$ and $e_4$, we take $e'_5:=e_5+\frac{J_{12}}{J_{14}} e_3$. In such a way  $[e'_5,e_6] \in \h_J$ and $\{e'_5,e_6\}$ define a new complement of $\h$ inside $\g$. Observe that since $\h$ is abelian, $A':=ad_{e'_{5}}\vert_{\h}=A$. \\
Let $e'_3$ and $e'_4$ be eigenvectors of $B$ associated to the eigenvalues $-2b_{11}$ and $0$, respectively. Since the eigenspaces $V_{-2b_{11}}$ and $V_0$ are $1$-dimensional and $[A,B]=0$, we must have  that $e'_3$ and $e'_4$ are also eigenvectors of $A$ with eigenvalue $0$. Hence, with respect the new basis $\{e_1,e_2,e'_3,e'_4,e'_5, e_6\}$ the Lie algebra $\g$ is determined by the data
\begin{equation*} 
\begin{split}  
&A'= \begin{pmatrix}
\begin{array}{c|c}
\begin{matrix}
0 & a_{12}\\
-a_{12} & 0
 \end{matrix} & 0\\
\hline
 0 & \begin{matrix} 0 & 0 \\
 0 & 0 \end{matrix} 
\end{array}
\end{pmatrix}, \ \
B= \begin{pmatrix}
\begin{array}{c|c}
\begin{matrix}
b_{11} & b_{12}\\
-b_{12} & b_{11} \end{matrix} & 0 \\
\hline
 0 & \begin{matrix} -2b_{11} &0 \\
0 & 0 \end{matrix} 
\end{array}
\end{pmatrix}, \ \ [e'_5,e_6]\in \h_{J}.
\end{split}
\end{equation*}
Finally,  let $X \in \h_J$ be such that $A'X=-[e'_5,e_6]$ (observe that it is possible since $\text{Im}A'=\h_J$). The basis $\{e'_1,\dots,e'_5,e'_6=e_6+X\}$ is such that $B:=ad_{e'_{6}}\vert_{\h}=B'$ and  $[e'_5,e'_6]=0$. The isomorphism between $\g$ and $\g_{5.35}^{-2,0} \oplus \R$ is immediate. 
\end{proof}
\begin{remark}
The Lie algebra $\g_{5.35}^{-2,0}$ firstly appears in \cite{BC}.
\end{remark}
\begin{remark}
The connected and simply connected solvable Lie groups corresponding to $\tau_{3,0} \times \tau_{3,0}$ and $\g_{5.35}^{-2,0} \oplus \R$ admit lattices. The former case is trivial. For the latter, an explicit construction of a lattice is done in the next section (Theorem \ref{thm:lattice}).
\end{remark}
\begin{corollary}
A  unimodular six-dimensional solvable Lie algebra with codimension $2$ nilradical $\h$ admits a generalized K\"ahler structure $(J_\pm, \langle \cdot,\cdot \rangle)$ if and only if $\g$ is isomorphic to one of the following Lie algebras
\begin{flalign*} 
&\begin{aligned}
 &\tau_{3,0} \times \tau_{3,0}=(-f^{25},f^{15},-f^{46},f^{36},0,0),\\
 &\g_{5.35}^{-2,0} \oplus \R=(2f^{15},-f^{25}-f^{36},-f^{35}+f^{26},0,0,0).
\end{aligned} &&
\end{flalign*}
An explicit  example of  generalized K\"ahler structure is given respectively by 
$$J_\pm f_1= \pm f_2, \ J_\pm f_3=\pm f_4, \ J_\pm f_5=\pm f_6, \  \langle \cdot,\cdot \rangle=\sum_{i=1}^6 (f^i)^2$$ 
$$J_\pm f_1=f_5, \ J_\pm f_2=\pm f_3, \ J_\pm f_4=f_6, \  \langle \cdot,\cdot \rangle=\sum_{i=1}^6 (f^i)^2.$$ 
\end{corollary}

\section{Construction of new SKT and generalized K\"ahler Lie algebras} \label{section5}
This section is devoted to construct new examples of SKT and generalized K\"ahler solvable Lie algebras. In particular, we exhibit examples of SKT solvable Lie algebras $(\g, \langle \cdot,\cdot \rangle,J)$ with $J\h=\h$ and  $J\h \neq \h$ in Examples \ref{ex:1}, \ref{ex:2}, and  in Examples \ref{ex:3}, \ref{ex:4}, respectively.
\begin{example} \label{ex:1}
By Theorem \ref{thm:classif}, we have that the only six dimensional SKT Lie algebra $(\g,  J, \langle \cdot,\cdot \rangle)$ such that $J\h=\h$ is  $\tau_{3,0} \times \tau_{3,0}$. Moreover, as $\tau_{3,0} \times \tau_{3,0}$ has trivial center, there exists a (possibly different) $J$-Hermitian structure which is K\"ahler, by Corollary \ref{thm:abelian} \emph{(iv)}. \\
In higher dimension, this is no longer true. Indeed, let us consider the eight dimensional solvable Lie algebra $\g^b_8, $ with $b \neq 0$, defined by the structure equations
\[
\begin{split}
&[e_1,e_7]=be_2, \quad [e_2,e_7]=-be_1,\quad [e_3,e_8]=be_4, \quad [e_{4},e_{8}]=-be_{3},\quad
[e_{7},e_{8}]=e_{5}+e_{6}, \\
\end{split}
\]
endowed with the complex structure $J$,  given by
\[
Je_{1}=e_{2}, \ \ Je_3=e_4, \ \ Je_5=e_6, \ \ Je_7=e_8.
\]
In particular, $\h=\text{span}_\R \langle e_1,\dots,e_6 \rangle$ and $J\h=\h$. Let us consider the inner product $\langle \cdot,\cdot \rangle$ with respect to which the basis  $\{e_1,\dots,e_{6}\}$ is orthonormal. Set $A:=ad_{e_{7}}\vert_\h$ and $B:=ad_{e_{8}}\vert_\h$. By Corollary \ref{thm:abelian} the Hermitian structure $J, (\langle \cdot, \cdot \rangle)$ is SKT and it is (non-flat) Chern Ricci flat by Theorem \ref{theorem:sktinvariant}. \\
We show that $(\g^b_8,J)$ does not admit any K\"ahler structure. 
Let $\{Z_1,\dots,Z_4\}$ be a unitary basis of $(\g^{b}_{8})^{(1,0)}$, with dual basis $\{\varphi^1,\dots,\varphi^4\}$. Then
\[
d\varphi^1=-\frac{ib}{2} (\varphi^{14}+\varphi^{1 \bar4}), \ d\varphi^2=-\frac{b}{2} (\varphi^{24}-\varphi^{2 \bar4}), \ \ d\varphi^3=-\frac{(1+i)}{2} \varphi^{4 \bar4}, \ \ d\varphi^4=0.
\]
Let us write the generic fundamental form
$
\omega=\frac{i}{2} \sum_{j=1}^4 F_{j \bar j} \varphi^{j \bar j}+\frac{1}{2} \sum_{j=1}^4 (F_{j \bar k} \varphi^{j \bar k}-\overline{F_{j \bar k}} \varphi^{k \bar j}),
$
with $F_{j \bar j} \in \R_{>0}$. Since $
d\omega(Z_3, Z_4, \overline Z_{4})=\frac{1-i}{2} F_{3\bar3},
$
we have that $d\omega \neq 0$.\\
\end{example}
\begin{theorem} \label{thm:newexamples}
Let $(\h, \langle \cdot , \cdot \rangle_\h, J_\h)$ be a SKT nilpotent Lie algebra and let $(\mathfrak{a}, J_\mathfrak{a},  \langle \cdot , \cdot \rangle_\mathfrak{a}, J_\mathfrak{a})$ an abelian Hermitian Lie algebra of dimension $2k$. Consider a Lie algebras homomorphism
\[
 \theta: \mathfrak{a} \to Der(\h)
\]
such that $\theta(\mathfrak{a}) \subset \mathfrak{sp}(\h) \cap  \mathfrak{so}(\h)$, namely $\theta(U)^t=-\theta(U)$ and $J_\h \theta(U)+\theta(U)^t J_\h=0$, $\forall U \in \mathfrak{a}$. Then  $(\g=\h \rtimes_\theta \mathfrak{a},\langle \cdot,\cdot \rangle, J)$,  with $\langle \cdot,\cdot \rangle=\langle \cdot,\cdot \rangle_\h+ \langle \cdot,\cdot \rangle_\mathfrak{a}$ and $J=\begin{pmatrix} J_\h & 0 \\
0 & J_\mathfrak{a}\\
\end{pmatrix}$,  is  SKT.
\end{theorem}
\begin{proof}
Let $\{U_1,J_\mathfrak{a} U_1,\dots,U_k,J_\mathfrak{a} U_k\}$ be an orthonormal basis of the abelian Hermitian Lie algebra $(\mathfrak{a}, J_\mathfrak{a},  \langle \cdot , \cdot \rangle_\mathfrak{a}, J_\mathfrak{a})$ with dual basis $\{u_1,J_\mathfrak{a} u_1,\dots,u_k, J_\mathfrak{a} u_k\}$. Observe that since $\g^1 \subset \h$, \ $du_i=d(J_\mathfrak{a} u_i)=0$. To lighten the notation, let us define $A_i:=\theta(u_i)$ and $B_i:=\theta(J_\mathfrak{a} u_i)$. \\
Firstly, we prove that $J$ is integrable. Exploiting that $N(J \cdot,J \cdot)=-N(\cdot,\cdot)$, one can see that it is enough to check the vanishing of $N(U_i,Y)$, for any $Y \in \h$ and $i=1,\dots,k$. We compute
\begin{align} \label{eqn:int3}
N(U_i,Y)&= [U_i,Y]+J_\h([J_\mathfrak{a} U_i,Y]+[U_i,J_\h Y])-[J_\mathfrak{a} U_i,J_\h Y]\\ \nonumber
&= (A_i+J_\h A_i J_\h+J_\h B_i-B_i J_\h) Y 
=[J_\h,A]J_\h+[J_\h,B] Y. \nonumber
\end{align}
As $A_i, B_i \in  \mathfrak{sp}(\h) \cap  \mathfrak{so}(\h)$, we get that  $0=J_\h A_i+A_i^t J_\h=J_\h A_i-A_i J_\h=[J_\h,A_i]$ and, analogously, $0=[J_\h,B_i]$. Hence \eqref{eqn:int3} vanishes, and $J$ is integrable. \\
Since $\h$ is a $J$-invariant ideal of $\g$ of even codimension, we may extend the formula \eqref{eqn:differential} in this case. Indeed, exploiting that $\mathfrak{a}$ is abelian, for any $\alpha \in \bigwedge^k \h^*$
\begin{equation} \label{eqn:differential2}
d\alpha= \sum_{i=1}^k u_i \wedge A_i^* \alpha+J_\mathfrak{a} u_i \wedge B_i^* \alpha +d_\h \alpha,
\end{equation}
where $d_\h$ stands for the differential of $\h$. \\
Let us consider the fundamental form $\omega$ of $(J, \langle \cdot,\cdot \rangle)$. Then, by construction, $$\omega=\omega_\h+\sum_{i=1}^k u_i \wedge J_\mathfrak{a} u_i.$$
Using \eqref{eqn:differential2}, we get that 
$
d\omega=d\omega_\h=\sum_{i=1}^k u_i \wedge A_i^* \omega_\h+J_\mathfrak{a} u_i \wedge B_i^*\omega_\h +d_\h \omega_\h.
$
Now, we know that $A_i, B_i \in  \mathfrak{sp}(\h) \cap  \mathfrak{so}(\h)$, so $\forall Y,Z \in \h$
\[
\begin{split}
A_i^* \omega_\h (Y,Z)&=-\omega_\h(A_i Y,Z)-\omega_\h(Y,A_iZ)=-\langle J_\h A_i Y,Z \rangle -\langle J_\h Y, A_i Z \rangle \\
&=-\langle (J_\h A_i + A_i^t J_\h) Y,Z \rangle =0,
\end{split}
\]
and, analogously, $B_i^* \omega_\h=0$. In particular, we get
$
d\omega=d_\h \omega_\h.
$
Let $c=Jd\omega=c_\h$ be the Bismut torsion. Again by \eqref{eqn:differential2},
$
dc=dc_\h=\sum_{i=1}^k u_i \wedge A_i^*c_\h+J_\mathfrak{a} u_i \wedge B_i^*c_\h,
$
as $d_\h c_\h$ vanishes since $( \langle \cdot , \cdot \rangle_\h, J_\h)$ is SKT. To conclude, we have to prove that $A_i^*c_\h=B_i^*c_\h=0$. We prove the statement for the former, since the latter is analogous. Let $Y,Z,W \in \h$. Then  
\[
\begin{split}
A_i^*c_\h(Y,Z,W)&=-c_\h(A_iY,Z,W)-c_\h(Y,A_iZ,W)-c_\h(Y,Z,A_iW) \\
&=\langle [J_\h A_i Y,J_\h Z], W \rangle + \langle [J_\h Z,J_\h W],A_i Y \rangle + \langle [J_\h W,J_\h A_i Y], Z \rangle \\
&+\langle [J_\h Y,J_\h A_i Z], W \rangle + \langle [J_\h A_iZ,J_\h W], Y \rangle + \langle [J_\h W,J_\h Y], A_iZ \rangle \\
&+\langle [J_\h Y,J_\h Z], A_i W \rangle + \langle [J_\h Z,J_\h A_iW], Y \rangle + \langle [J_\h A_iW,J_\h Y], Z \rangle.
\end{split}
\]
Exploiting that $[A_i,J_\h]=0$ and $A_i^t=-A_i$, one gets
\[
\begin{split}
A_i^*c_\h(Y,Z,W)&=\langle -A_i [J_\h Y,J_\h Z],W \rangle + \langle [ A_i J_\h Y,J_\h Z], W \rangle +\langle [J_\h Y,A_i J_\h Z], W \rangle  \\
& + \langle -A_i [J_\h Z,J_\h W],Y \rangle  + \langle [ A_iJ_\h Z,J_\h W], Y \rangle +  \langle [J_\h Z, A_i J_\h W], Y \rangle \\
&  + \langle-A_i [J_\h W,J_\h Y], Z \rangle + \langle [ A_i J_\h W,J_\h Y], Z \rangle+  \langle [J_\h W,A_i J_\h Y], Z \rangle,\\
\end{split}
\]
which vanishes since $A_i$ is a derivation of $\h$. 
\end{proof}
\begin{example} \label{ex:2} Six dimensional (non abelian) nilpotent Lie algebras admitting a SKT structure are classified  in \cite{FPS04} (see also  \cite[Theorem 3.3]{UG}). Let $\h$ be a six-dimensional nilpotent Lie algebra defined by the structure equations
\begin{equation} \label{eqn:eqstr}
de^1=de^2=de^3=de^4=0, \ de^5=\rho e^{13}-\rho e^{24}+2\delta e^{34}, \ de^6=\rho e^{23}+\rho e^{14}-2e^{12}-2\gamma e^{34},
\end{equation} 
with $\rho\in \{0,1\}, \ \delta,\gamma \in \R$ and $\rho^2-2\gamma=0$. A SKT structure is given by 
\begin{equation} \label{eqn:sktstructure}
 J_\h e_1=e_2, \ J_\h e_3=e_4, \ J_\h e_5=e_6, \ \langle \cdot,\cdot \rangle_\h=\sum_{i=1}^6 (e^i)^2.
\end{equation}
Four different non-isomorphic Lie algebras are distinguished
\begin{equation} \label{eqn:ugarteclassification}
\begin{split}
&\rho=0,\  \gamma=0 \implies \ \begin{cases}  \h \cong \h_2=(0,0,0,0,12,34) \ \ \text{for} \ \delta \neq 0, \\
\h \cong \h_8=(0,0,0,0,0,12)  \ \ \text{for}\ \delta=0,
\end{cases} \\
&\rho=1, \gamma=\frac{1}{2} \implies 
\begin{cases}  \h \cong \h_2=(0,0,0,0,12,34) \ \  \text{for} \ 4\delta^2 > 3, \\
\h \cong \h_4=(0,0,0,0,12,14+23) \ \  \text{for} \ 4\delta^2 = 3, \\
\h \cong \h_5=(0,0,0,0,13+42,14+23) \ \  \text{for} \ 4\delta^2 < 3.
\end{cases} 
\end{split}
\end{equation}
By \cite[Theorem 3.2]{FPS04}  these are the only possible six dimensional nilpotent Lie algebras admitting a SKT structure, up to isomorphism. \\
We want to exploit Theorem \ref{thm:newexamples} to extend the Lie algebras described above to $8$-dimensional solvable SKT Lie algebras. 
Let $\h$ be a six dimensional nilpotent Lie algebra with structure equations  \eqref{eqn:eqstr}  endowed with the SKT structure \eqref{eqn:sktstructure}. Consider the abelian Lie algebra $\mathbb R^2$ endowed with the standard Hermitian structure and let $\{U, U' \}$ be a unitary basis of $\R^2$ satisfying $U'=JU$  with dual basis $\{ u,u' \}$. We define the Lie algebra homomorphism $\theta: \R^2 \to Der(\h)$
as  
\[ 
\theta(U)=\begin{pmatrix}
\begin{array}{cc|cc|cc}
0 & a & 0 & 0 & 0 & 0 \\
-a& 0 & 0 & 0& 0& 0\\ \hline
0& 0& 0 & -a & 0 & 0 \\
0 & 0 & a & 0 & 0 & 0 \\ \hline
0 & 0 & 0 &0  & 0 & 0 \\
0 &0  &0 & 0 & 0 & 0
\end{array}
\end{pmatrix}
\]
and $\theta(U')=0$, with respect to the fixed basis $\{e_1,\dots,e_6\}$ of $\h$. One can prove that $\theta(\R^2) \subset Der(\h)$ by direct computation. Since $\theta(U)$ (and trivially $\theta(U')$) is in $\mathfrak{sp}(\h) \cap \mathfrak{so}(\h)$ with respect to the diagonal metric $\langle \cdot,\cdot \rangle_\h=\sum_{i=1}^6 (e^i)^2$, Theorem \ref{thm:newexamples} applies. Hence, the Lie algebra $(\g=\h \rtimes_\theta \R^2, J, \langle \cdot,\cdot \rangle)$ with $\langle \cdot,\cdot \rangle=\sum_{i=1}^6 (e^i)^2+u^2+u'^2 $ and 
$J=\begin{pmatrix} \begin{array}{ c|cc}
J_\h & & \\ \hline
& 0 & -1 \\
& 1 & 0 \\
\end{array}
\end{pmatrix}$ is SKT.\\ \ \\
In particular, the following families of $8$-dimensional decomposable Lie algebras 
\begin{equation} \label{eqn:newskt} 
\begin{aligned}
 &\s^{a, \delta}_{1}=(a f^{27},-a f^{17},-a f^{37},a f^{47},2\delta f^{34},-2f^{12},0,0),  \\ 
 & \ \quad  \qquad\  \text{with} \ \delta \neq 0, \\ 
 &\s^a_{2}=(a f^{27},-a f^{17},-a f^{37},a f^{47},0,-2f^{12},0,0),  \\ 
&\s^{a, \delta}_{3} =(a f^{27},-a f^{17},-a f^{37},a f^{47},f^{13}-f^{24}+2\delta f^{34},f^{23}+f^{14}-2f^{12}-f^{34},0,0) ,	\\
&\ \quad \qquad \  \text{with} \ 4\delta^2> 3, \\
&\s^a_{4}=(a f^{27},-a f^{17},-a f^{37},a f^{47},f^{13}-f^{24}+\sqrt{3}f^{34},f^{23}+f^{14}-2f^{12}-f^{34},0,0),   \\ 		
&\s^{a,\delta}_{5}=(a f^{27},-a f^{17},-a f^{37},a f^{47},f^{13}-f^{24}+2\delta f^{34},f^{23}+f^{14}-2f^{12}-f^{34},0,0) , \\
& \ \quad  \qquad\  \text{with} \ 4\delta^2< 3
\end{aligned}
\end{equation}
are extensions via the homomorphism $\theta$ of the nilpotent Lie algebras listed in \eqref{eqn:ugarteclassification} and admit a SKT structure. More precisely, $\s^{a, \delta}_{1},\s^{a, \delta}_{3}$ are extensions of $\h_2$ and $\s^a_{2}, \s^a_{4}, \s^{a,\delta}_{5}$ are extensions of $\h_8,\h_4,\h_5$, respectively. Observe that the Lie algebras above are nilpotent for $a=0$ and almost nilpotent otherwise. Examples such that the nilpotent Lie algebra $\h$ coincides with the nilradical of the extension can be constructed in the same way as above by setting $\rho=\gamma=0$ in \eqref{eqn:eqstr} and 
\[ 
\theta(U)=\begin{pmatrix}
\begin{array}  {cc|cc|cc}
0 & a  & 0 & 0 & 0 & 0\\
-a& 0 & 0 & 0 & 0& 0\\ \hline
0 & 0 & 0 & 0 &0 &0\\
0 &0 & 0 & 0 & 0 &0 \\ \hline
0 & 0 &0 & 0 & 0 & 0 \\
0&0 &0 &0 & 0 & 0
\end{array}
\end{pmatrix}, \ \ 
\theta(U')=\begin{pmatrix}
\begin{array}{cc|cc|cc}
0 & 0  &0 & 0& 0&0\\
0& 0 & 0&0 &0 &0 \\ \hline
0&0 & 0 & b & 0 &0 \\
0& 0& -b & 0 & 0 &0\\ \hline
0&0 & 0&0 & 0 & 0 \\
0&0 &0 &0 & 0 & 0
\end{array}
\end{pmatrix}
\]
with $a,b \neq 0$. 
Therefore, the following families of $8$-dimensional indecomposable solvable Lie algebras
\begin{flalign*} 
&\begin{aligned}
 &\s^{a,b,\delta}_{6}=(a f^{27},-a f^{17},b f^{38},-b f^{48},2\delta f^{34},-2f^{12},0,0),  \quad a,b, \delta \neq 0,   \\
 &\s^{a,b}_{7}=(a f^{27},-a f^{17},b f^{38},-b f^{48},0,-2f^{12},0,0), \quad  a,b \neq 0,   \end{aligned} && 
\end{flalign*}
have nilradical $\h_2$ and $\h_8$ respectively, and admit a SKT structure $(J, \langle \cdot,\cdot \rangle)$, given by
\[
 \langle \cdot, \cdot \rangle =\sum_{i=1}^8 (f^i)^2, \ Jf_1=f_2, \ Jf_3=f_4, \ Jf_5=f_6, \ Jf_7=f_8.
\] 
such that $J\h=\h$.
We can prove  now that $(\s^{a,b,\delta}_{6}, J)$ and $(\s^{a,b}_{7}, J)$ do not admit any balanced metric. The proof proceeds in the same way for both Lie algebras.  Assume that there exists a balanced $J$-Hermitian metric $\langle \cdot,\cdot \rangle_2$. Since $(\h, J)$ admits an SKT metric, then its center $\mathfrak{z}(\h)$ is $J$-invariant, and so we have the decomposition of the Lie algebra as $\mathfrak{z}(\h)^{\perp_{\langle, \rangle_2}} \oplus \mathfrak{z}(\h) \oplus \h^{\perp_{\langle, \rangle_2}}$ with respect to $\langle \cdot,\cdot \rangle_2$. We fix an unitary basis $\{e_1,\dots,e_8\}$ for the metric $\langle \cdot,\cdot \rangle_2$, such that $\{e_1,\dots,e_4\}$ is a unitary basis of $\mathfrak{z}(\h)^{\perp_{\langle, \rangle_2}}$ satisfying $Je_{2i-1}=e_{2i}$, $\{e_5,e_6\}$ is a unitary basis of $\mathfrak{z}(\h)$ satisfying $Je_5=e_6$ and  $\{e_7,e_8\}$ is a unitary basis of $\h^{\perp_{\langle, \rangle_2}}$ satisfying $Je_7=e_8$. Since  $(J, \langle \cdot,\cdot \rangle_2)$ is balanced, we have that the Lee form $\theta$, defined as 
\[
\theta(X)=\frac{1}{2} \langle \sum_{i=1}^8 [e_i,Je_i], JX \rangle_2=\frac{1}{2} \langle \sum_{i=1}^4 [e_i,Je_i], JX \rangle_2 + \langle  [e_7,e_8], JX \rangle_2,
\]
is identically zero. Observe that since $\h$ is $2$-step nilpotent, $\sum_{i=1}^4 [e_i,Je_i] \subset \mathfrak{z}(\h)$. Let $X$ be any vector of $\mathfrak{z}(\h)^{\perp_{\langle, \rangle_2}}$. Then $ \theta(X)= \langle  [e_7,e_8], JX \rangle_2=0$, impliying that $ [e_7,e_8] \in \mathfrak{z}(\h)$. In terms of the previous basis $\{f_1,\dots,f_8\}$, we may write $e_7=\lambda f_7+ \mu f_8 + Y$, with $Y \in \h$, and 
$e_8=Je_7=\lambda f_8- \mu f_7 + JY$. Exploiting that $[f_7,f_8]=0$, we get $$[e_7,e_8] =  [f_7, \lambda JY + \mu Y]+[f_8, \mu JY -\lambda Y]+[Y,JY]\in \big(\Im(ad_{f_{7}} \vert_\h) \oplus  \Im(ad_{f_{8}} \vert_\h) \oplus \mathfrak{z}(\h)\big) \cap \mathfrak{z}(\h).$$ Hence, the components of $[e_7,e_8]$ along $\Im(ad_{f_{7}} \vert_\h)$ and $  \Im(ad_{f_{8}} \vert_\h)$ must vanish. Moreover, a straightforward computation shows that this forces $\lambda=\mu=0$, giving a contradiction. 
\end{example}
\begin{example} \label{ex:3}
Now, we give an example of a $8$-dimensional solvable Lie algebra with codimension $2$ nilradical $\h \cong \h_8$, endowed with a SKT structure $(\langle \cdot,\cdot \rangle, J)$ such that $J\h\neq \h$. \\
Consider the $8$-dimensional solvable Lie algebra $\s_{8}$ defined by the structure equations
\[
\begin{split}
&de^1=e^{23}, \ de^2=e^{27}, \ de^3=-e^{37}, de^4=e^{57}+e^{48},\\
&de^5=-e^{47}+e^{58}, \ de^6=-2e^{68}, \ de^7=de^8=0,
\end{split}
\]
with nilradical $\{e_1,\dots,e_{6}\} \cong \h_8$. We define the Hermitian structure $(\langle \cdot,\cdot \rangle, J)$ as
\[
\langle \cdot,\cdot \rangle=\sum_{i=1}^8 (e^i)^2, \ Je_1=e_2, \ Je_3=e_7, \ Je_4=e_5, \ Je_6=e_8.
\]
Then, $J\h \neq \h$ and the associated fundamental form is
\[
\omega=e^1 \wedge e^2 +e^3 \wedge e^7 +e^4 \wedge e^5 +e^6 \wedge e^8.
\]
Since the Bismut torsion $c=Jd\omega=-e^{123}-2e^{456}$ is closed, the Hermitian structure $(\langle \cdot,\cdot \rangle, J)$ is SKT. In Theorem \ref{theorem:sktinvariant}, we have proved that when the nilradical $\h$ is $J$-invariant, the Hermitian structure is Chern Ricci flat. It is no longer true when the nilradical is not $J$-invariant, in fact, if one compute the Chern connection of the Lie algebra above, then $\eta^{Ch}=e^{3}+e^6+e^7 $ and $\rho^{Ch}=- e^{37}-2e^{68} \neq 0$.\\
Let $\{\varphi^{1}:=e^1 + ie^2,\varphi^{2}:=e^3+ie^7,\varphi^3:=e^4+ie^5,\varphi^4:=e^6+ie^8\}$. An easy computation shows that $d\varphi^2=\frac{i}{2} \varphi^{2 \overline{2}}$. We use $\varphi^2$ to prove that $(\mathfrak{s}_8, J)$ does not admit any balanced metric. Assume by contradiction that there exists a balanced metric with associated fundamental form $\sigma$. Then $\varphi^2 \wedge \sigma^3 \neq 0$. The differential $d(\varphi^2 \wedge \sigma^3 )=\frac{i}{2} \varphi^{2 \overline{2}} \wedge \sigma^3$ is non zero, and so we have a (non-zero) $2n-1$ form on a unimodular Lie algebra which is not closed. The contradiction follows.
\end{example}
\begin{example} \label{ex:4}
Consider the solvable Lie algebra $\g_{b,c,c'}^{2n}$, \ $b,c,c' \in \R \setminus\{0\}$, defined for $n \ge 3$ by the structure equations
\[
\begin{cases}
& de^1=-e^1 \wedge e^{2n-1}, \\
& de^2=\frac{1}{2}e^2 \wedge e^{2n-1}+ b e^3 \wedge e^{2n-1}-c e^3 \wedge e^{2n}, \\
& de^3=-be^2 \wedge e^{2n-1}+ \frac{1}{2} e^3 \wedge e^{2n-1}+c e^2 \wedge e^{2n}, \\
& de^{2l}=c' e^{2l+1}\wedge e^{2n}, \ \  de^{2l+1}=-c' e^{2l}\wedge e^{2n} \ \  \text{for $l=2,\dots,n-2$}, \\
& de^{2n-2}=0, \\
& de^{2n-1}=0, \\
& de^{2n}=0.
\end{cases}
\]
Observe that for $n=3$, $\g_{b,c,c'}^{6} \cong \g_{5.35}^{-2,0} \oplus \R$.  If we denote by $\{e_1,\dots,e_{2n}\}$ the dual basis, the codimension $2$-abelian nilradical $\h$ is spanned by $ \{e_1,\dots,e_{2n-2}\}$, and $\g_{b,c,c'}^{2n}\cong (\R^{2n-3}\rtimes \R^2) \times \R$, where $\R^{2n-3}, \R^2$ and $\R$ are spanned by $\{e_{1},\dots,e_{2n-3}\}$, $\{e_{2n-1},e_{2n}\}$ and $e_{2n-2}$, respectively. 
Let us define the bi-Hermitian structure $(I_{\pm}, \langle \cdot,\cdot \rangle)$ as
\[
I_\pm e_1=e_{2n-1}, \  I_\pm e_2= \pm e_{3},\  I_\pm e_{2l}=e_{2l+1}\  \text{for $l=2,\dots,n-2$},\  I_\pm e_{2n-2}=e_{2n}, \  \langle \cdot,\cdot \rangle=\mathlarger{\mathlarger{\sum}}_{i=1}^{2n} (e^i)^2. \ 
\]
It is straightforward to note that $I_\pm \h\neq \h$. The corresponding fundamental forms $\omega_\pm=e^1 \wedge e^{2n-1} \pm e^2 \wedge e^3 + e^{2n-2}\wedge e^{2n}+\mathlarger{\mathlarger{\sum}}_{l=2}^{2n-1} e^{2l}\wedge e^{2l+1}$ satisfy
$
d^c_\pm\omega_\pm= \pm \ e^1 \wedge e^2 \wedge e^3,
$
and by the structure equations $e^{1}\wedge e^{2}\wedge e^{3}$ is a closed $3$-form, i.e., $( I_\pm, \langle \cdot,\cdot \rangle)$ is a generalized K\"ahler structure of split-type, i.e. $I_{\pm}$ commute. 
If  $n \ge 5$,  then $\g^{2n}_{b,c,c'}$ admits also a non-split generalized K\"ahler structure. Let the bi-Hermitian structure $( I_\pm, \langle \cdot,\cdot \rangle)$ be defined as
\[
\begin{split}
&\langle \cdot,\cdot \rangle=\mathlarger{\mathlarger{\sum}}_{i=1}^{2n} (e^i)^2, \\
&I_+ e_1=e_{2n-1},  \ I_+ e_2= + e_{3}, \ I_+ e_{4}=e_{5}, \\
& I_+ e_{6}=-e_{7},  \ I_+ e_{2l}=e_{2l+1}\  \text{for $l=4,\dots,n-2$}, \ I_+ e_{2n-2}=e_{2n}, \\
&I_- e_1=e_{2n-1},  \ I_- e_2= - e_{3}, \ I_- e_{4}=-e_{7}, \ I_- e_{5}=-e_{6}, \\
& I_- e_{2l}=e_{2l+1}\  \text{for $l=4,\dots,n-2$}, \ I_- e_{2n-2}=e_{2n}.\\
\end{split}
\]
Since 
$
I_+ I_- e_4=-I_+ e_7=-e_6, \ \ \text{and} \ \ I_- I_+ e_4=I_- e_5=e_6,
$
the commutator $[I_+,I_-] \neq 0$. \\
Let $\omega_\pm$ be the fundamental forms associated to $( I_\pm, \langle \cdot,\cdot \rangle)$, which are respectively
\[
\omega_+=e^1 \wedge e^{2n-1}+e^2 \wedge e^3+e^4 \wedge e^5 -e^6 \wedge e^7 + e^{2n-2}\wedge e^{2n}+ \mathlarger{\mathlarger{\sum}}_{l=4}^{2n-1} e^{2l}\wedge e^{2l+1},
\]
and $\omega_-=e^1 \wedge e^{2n-1}+e^2 \wedge e^3-e^4 \wedge e^7 +e^5 \wedge e^6 + e^{2n-2}\wedge e^{2n}+ \mathlarger{\mathlarger{\sum}}_{l=4}^{2n-1} e^{2l}\wedge e^{2l+1}.$
Then, as before, $
d^c_\pm \omega_\pm = \pm \ e^1 \wedge e^2 \wedge e^3$ and  $d ( e^1 \wedge e^2 \wedge e^3)=0$, 
implying that $(I_\pm, \langle \cdot,\cdot \rangle)$ is a generalized K\"ahler structure of non-split type.
\end{example}

\section{Results on compact  solvmanifolds} 
Using the Symmetrization process (\cite{BE,FG,UG}),  one can prove that if $M=\Gamma \backslash G$ is a compact  solvmanifold endowed with a pair of invariant complex structures $J_\pm$, then $M$ admits a generalized K\"ahler structure $(J_\pm, g)$, if and only if the Lie algebra $\g$ of $G$ admits a generalized K\"ahler inner product $(J_\pm, \langle \cdot,\cdot \rangle)$. Hence, we may prove the following results
\begin{theorem}
Let $M=\Gamma \backslash G$ be a $2n$-dimensional solvmanifold and let $J_\pm$ be invariant complex structures on $M$.  Assume that the nilradical $\h$ of the Lie algebra $\g$ of $G$ has codimension $2$.
\begin{enumerate}
\item[(i)] If $J_\pm \h=\h$, then $(M,J_\pm)$ admits a generalized K\"ahler metric if and only if  $(M, J_+)$ admits a K\"ahler metric.
\item[(ii)] Assume that $\h$ is abelian and the complex structures $J_\pm$ satisfy $J_+\h\neq\h$ and $J_-\h=\h$. If $(M,J_\pm)$ admits a generalized K\"ahler metric, then $(M, J_-)$ admits also a K\"ahler metric.
\end{enumerate} 
\end{theorem}
\begin{proof}
To prove \emph{(i)}, we observe that the implication from right to left is obvious. Let us prove the  other implication. By the Symmetrization process, $M=\Gamma \backslash G$ admits a generalized K\"ahler structure  $(J_{\pm}, g)$ if and only if $\g$ admits a generalized K\"ahler structure $(J_{\pm}, \langle \cdot,\cdot \rangle)$. Then the result follows by applying Theorem \ref{thm:equivalent}. \\
To prove \emph{(ii)}, again by the Symmetrization process, there must exists a left invariant generalized K\"ahler structure $(\langle \cdot,\cdot \rangle, J_\pm)$ on the Lie algebra $\g$ with $J_\pm$ satisfying $J_+\h\neq\h$ and $J_-\h=\h$. Then the statement follows by applying Proposition \ref{proposition:gk}.
\end{proof}
\begin{corollary}
Let $M=\Gamma \backslash G$ be a $2n$-dimensional solvmanifold and let $I_\pm$ be left invariant complex structures on $M$.  Assume that the nilradical $\h$ of the Lie algebra $\g$ of $G$ is abelian and of codimension $2$. If $(M,I_\pm)$ admits a generalized K\"ahler metric, then $I_\pm \h \neq \h$.
\end{corollary}
Now, we exhibit new examples of SKT and generalized K\"ahler  compact solvmanifolds. 
\begin{example}
Let $\g^b_8$ be the $8$-dimensional solvable Lie algebra constructed in Example \ref{ex:1}. We prove that for $b=2\pi$, the simply conncted Lie group with Lie algebra $\g^{2\pi}_{8}$ admits a compact quotient. Let $G^{2\pi}_{8}$ be the corresponding connected and simply connected Lie Group with group operation $\ast$ given by 
\[
\tiny
\begin{pmatrix}
a_1 \\
a_2 \\
a_3 \\
a_4 \\
a_5 \\
a_6\\
a_7\\
a_8
\end{pmatrix}
\ast
\begin{pmatrix}
x_1 \\
x_2 \\
x_3 \\
x_4 \\
x_5 \\
x_6\\
x_7\\
x_8
\end{pmatrix}
= 
\begin{pmatrix}
a_1 \\
a_2 \\
a_3 \\
a_4 \\
a_5 \\
a_6\\
a_7\\
a_8
\end{pmatrix}+
\begin{pmatrix}
\cos (2\pi a_7) & -\sin(2\pi a_7) & 0 & 0 & 0 & 0 &0 & 0\\
\sin(2\pi a_7) & \cos(2\pi a_7) & 0 & 0 & 0 & 0 & 0 & 0\\
0 & 0 & \cos (2\pi a_8) & -\sin(2\pi a_8) &  0 & 0 & 0 & 0\\
0 & 0 & \sin(2\pi a_8) & \cos(2\pi a_8) &  0 & 0 & 0 & 0\\
0 &0 & 0& 0& 1 & 0 & 0 & -a_7 \\
0&0 & 0& 0& 0 & 1 & 0 & -a_7 \\
0&0 &0 &0 & 0 & 0 & 1 & 0 \\
0& 0&0 &0 & 0 & 0 & 0 & 1 \\
\end{pmatrix}
\cdot 
\begin{pmatrix}
x_1 \\
x_2 \\
x_3 \\
x_4 \\
x_5 \\
x_6\\
x_7\\
x_8
\end{pmatrix}.
\] 
If we consider $\Gamma=\{(a_1,\dots,a_8) \in G^{2\pi}_{8}  \ | \ a_i \in  \Z\} \cong \Z^8$, then it is straightforward to note that $\Gamma$ is a discrete subgroup of $G^{2\pi}_{8}$ of maximal rank. The corresponding solvmanifold $(\Gamma \backslash G^{2\pi}_{8},J)$ admits a SKT (non-flat) Chern Ricci flat metric.  Note that $(\Gamma \backslash G^{2\pi}_{8},J)$ does not admit any balanced metric. Indeed, if $(\Gamma \backslash G^{2\pi}_{8},J)$ admits a balanced metric, then by the symmetrization process, $(\g^{2\pi}_8, J)$ admits a balanced metric, moreover since the Lie algebra has an abelian ideal of codimension $2$, $\g^{2\pi}_8$ would also admit a K\"ahler metric (\cite{CZ}), giving a contradiction.
\end{example} 
\begin{example}
Let $\s$ be any of the Lie algebras listed in \eqref{eqn:newskt} (Example \ref{ex:2}). For $a=0, 2\pi$ the corresponding connected and simply connected Lie groups admit compact quotients.	\\
For $a=0$ the statement is trivial. Indeed, in this case $\s$ is isomorphic to $\h \oplus \R^2$. Moreover, since $\h$ is isomorphic to either  $\h_2,\h_4,\h_5,\h_8$ depending on the values of $\rho,\gamma,\delta$, we always have that $\h$ admits a basis with rational structure constants. \\
For $a=2\pi$, since $\s \cong (\h \rtimes_A \R \langle X \rangle) \oplus \R$, we may restrict to prove that $\g=\h \rtimes_A \R \langle X \rangle$ admits compact quotients for $a=2\pi$. In particular, $\h$ is the nilradical of  $\h \rtimes_A \R \langle X \rangle$.  Let $G=H \rtimes_\mu \R$ be the corresponding connected and simply connected almost nilpotent Lie group. \\
By \cite{BC}, if there exists $0 \neq t_1 \in \R$ and a rational basis $\{Y_1,\dots,Y_6\}$ of $\h$ such that the coordinate matrix of $d_e \mu(t_1)=\text{exp}(t_1 A)$ in such a basis is integer, then $G$ admits a lattice $\Gamma$ of splitting type, i.e., $\Gamma=\Gamma_H \rtimes \Gamma_\R$. 
Let $t_1=1$, then 
$
d_e \mu(1)=\text{exp}(A)= Id_6.
$
Hence, $d_e \mu(1)$ is the identity of $\h$ and its coordinate matrix is integer for any chosen basis of $\h$. In particular, this is true for any rational basis $\{Y_1,\dots,Y_6\}$ of $\h$, which exists by previous observations. 
Therefore, the connected and simply connected solvable Lie group $G \times \R \cong H \rtimes \C$ corresponding to $\s $ admits a lattice of splitting type $\Gamma' = (\Gamma_H \rtimes \Gamma_\R ) \times \Z= \Gamma_H \rtimes \Gamma_\C$, where $\Gamma_\C= \Gamma_\R \times \Z$. \\
We also claim that the solvmanifold  $\Gamma' \backslash H \rtimes \C$ does not admit any balanced metric. Indeed, one can easily prove that the complex structure $J$ on $\Gamma' \backslash H \rtimes \C$ is of splitting type (see \cite{AOUV} for further details). The only non trivial check is that the Dolbeault cohomology of $(\Gamma_H \backslash H, J_\h)$  is the left invariant one. However, this follows by \cite[Corollary 3.10] {SR}. \\
Now, assume by contradiction that $(\Gamma' \backslash H \rtimes \C, J)$ admits a balanced metric. Then by \cite[Proposition 2.1]{AOUV}, $(\h, J_\h)$ admits a balanced metric, and so $(\h, J_\h)$ admits both a SKT and a balanced metric. By \cite{FV}, $\h$ would be abelian, giving a contradiction. 
\end{example}
For the Lie algebras constructed in Example \ref{ex:4}, we have the following 
\begin{theorem} \label{thm:lattice} 
Let  $G^{2n}_{b,c,c'}$ be the connected and simply connected Lie group corresponding to the Lie algebra $\g^{2n}_{b,c,c'}$ constructed in Example \ref{ex:4}. $G^{2n}_{b,c,c'}$ admits a lattice $\Gamma$ for some values of $b,c,c' \in \R \setminus \{0\}$ and, denoted by $M$ the corresponding compact solvmanifold, $M$ admits generalized K\"ahler structure of split type for $n \ge 3$ and of non-split type for $n \ge 5$. Furthermore, we have that $b_1(M)=3$. In particular, $M$ does not admit any K\"ahler metric.
\end{theorem}
\begin{proof}
The  Lie group $G^{2n}_{b,c,c'}$ underlying the Lie algebra $\g^{2n}_{b,c,c'}$ is the semi-direct product $\R^{2n-2}\rtimes_\phi \R^2$. In the following we will denote by $\alpha=\,^t(\alpha_1,\dots,a_{2n-2})$ and $(a_{2n-1},a_{2n})$ the coordinates on $\R^{2n-2}$ and $\R^2$, respectively. The multiplication $\ast$ is defined as
\[  \tiny
\begin{pmatrix}
\alpha\\
a_{2n-1}\\
a_{2n}\\
\end{pmatrix} \ast 
\begin{pmatrix}
\underline{x}\\
x_{2n-1}\\
x_{2n}\\
\end{pmatrix} =
\begin{pmatrix}
\alpha+ \phi(a_{2n-1},a_{2n})\cdot \underline{x}\\
a_{2n-1}+x_{2n-1}\\
a_{2n}+x_{2n}\\
\end{pmatrix} 
\]
where $\phi$ is the diagonal block matrix
\[ \tiny
\phi(a_{2n-1},a_{2n}) \cdot \underline{x} =
\begin{pmatrix}
e^{a_{2n-1}}\cdot x_1 \\
e^{-{\frac{a_{2n-1}}{2}}} R(ca_{2n}-ba_{2n-1}) \cdot \begin{pmatrix} 
x_2 \\
x_3 \\
\end{pmatrix} \\
\vdots \\
R(c'a_{2n})\cdot 
\begin{pmatrix}
x_{2l} \\
x_{2l+1} \\
\end{pmatrix} \\
\vdots \\
x_{2n-2} 
\end{pmatrix}
\]
with 
\[ \tiny
R (\theta)=\begin{pmatrix}
\cos(\theta) & \sin(\theta) \\
-\sin(\theta) & \cos(\theta) \\
\end{pmatrix}.
\]
Consider the $3\cross3$ square matrix 
\[ \tiny
\begin{split}
A(a_{2n-1},b)&=
\begin{pmatrix}
e^{a_{2n-1}} & 0 \\
0 & e^{-{\frac{a_{2n-1}}{2}}} \cdot R(-ba_{2n-1}) 
\end{pmatrix} =\exp \left( a_{2n-1} \cdot \begin{pmatrix} 
1 & 0 & 0 \\
0 &-\frac{1}{2} & -b \\
0 & b & -\frac{1}{2} \\
\end{pmatrix}
\right).
\end{split}
\]
By \cite{AO} there exist $t_0, \overline{b} \in \R \setminus \{0\}$ such that  $A:=A(t_0,\overline{b})$ is similar to an integer matrix $\Lambda \in SL(3,\Z)$ via an invertible matrix $P=(p_{ij})$, i.e., there exists an invertible matrix $P$ such that $AP=P\Lambda$. It is then easy to observe that for each $n \in \Z$, $A^n=P\Lambda^n P^{-1}$, where by a straightforward computation $A^n=A(t_0n,\overline{b}).$ \\
If we set $b=\overline{b}$, $c=\frac{2\pi}{t_{0}}$ and $c'=\frac{\pi}{2t_0}$, we claim that the discrete set $\Gamma:= P \Z^3 \times (\Z^2)^{n-3} \times \Z \times t_0 \Z^2 $ is a lattice of $G^{2n}_{\overline{b},\frac{2\pi}{t_{0}},\frac{\pi}{2t_{0}}}$. 
We first observe that for such values of $b,c,c'$ 
\[  \tiny
\phi(t_0 n,t_0m) \cdot \underline{x} =
\begin{pmatrix}
A^n \cdot \begin{pmatrix} 
x_1 \\
x_2 \\
x_3 \\
\end{pmatrix} \\
\vdots \\
R\left(\frac{\pi}{2}\cdot m \right)\cdot 
\begin{pmatrix}
x_{2l} \\
x_{2l+1} \\
\end{pmatrix} \\
\vdots \\
x_{2n-2} 
\end{pmatrix}
\]
Then, fixed $\,^t(Pz,w_1,\dots,w_{n-3},q,t_0n,t_0m)$ and $\,^t(Pz',w'_1,\dots,w'_{n-3},q',t_0n',t_0m') \in \Gamma$ \[ \tiny
\begin{pmatrix}
Pz \\
w_1 \\
\vdots \\
w_{n-3} \\
q \\
t_0n \\
t_0m
\end{pmatrix}
\ast
\begin{pmatrix}
Pz' \\
w'_1 \\
\vdots \\
w'_{n-3} \\
q \\
t_0n' \\
t_0m'
\end{pmatrix} ^{-1}= \begin{pmatrix}
P(z -\Lambda^{(n-n')}\cdot z' )\\
w_1- R(\frac{\pi}{2}\cdot (m-m')) w_1' \\
\vdots \\
w_{n-3}-R(\frac{\pi}{2}\cdot (m-m')) w_{n-3}'  \\
q-q' \\
t_0(n-n') \\
t_0(m-m') \\
\end{pmatrix}
\]

\begin{comment}
\begin{pmatrix}
Pz \\
w_1 \\
\vdots \\
w_{n-3} \\
q' \\
t_0n \\
t_0m
\end{pmatrix}
\ast
\begin{pmatrix}
\phi(-t_0n',-t_0m') \cdot \begin{pmatrix} 
-Pz' \\
-w'_1 \\
\vdots \\
-w'_{n-3} \\
-q 
\end{pmatrix} \\
-t_0n' \\
-t_0m'
\end{pmatrix}\\ 
&=\begin{pmatrix}
\begin{pmatrix}
Pz \\
w_1 \\
\vdots \\
w_{n-3} \\
q' \\
\end{pmatrix} +
\begin{pmatrix}
\phi(t_0(n-n'),t_0(m-m')) \cdot \begin{pmatrix} 
-Pz' \\
-w'_1 \\
\vdots \\
-w'_{n-3} \\
-q 
\end{pmatrix}
\end{pmatrix} \\
t_0(n-n') \\
t_0(m-m') \\
\end{pmatrix} \\
&=
\begin{pmatrix}
\begin{pmatrix}
Pz \\
w_1 \\
\vdots \\
w_{n-3} \\
q \\
\end{pmatrix}  -
\begin{pmatrix}
A^{(n-n')} \cdot P z' \\
R(\frac{\pi}{2}\cdot (m-m')) w_1' \\
\vdots \\
R\left(\frac{\pi}{2}\cdot (m-m') \right) w_{n-3}' \\
q' \\
\end{pmatrix} \\
t_0(n-n') \\
t_0(m-m') \\
\end{pmatrix} \\
\end{split}
\]
Using the fact that $A^r P=P \Lambda^r$ for each integer $r$ we find
\[ \tiny
\begin{split}
\begin{pmatrix}
\begin{pmatrix}
Pz \\
w_1 \\
\vdots \\
w_{n-3} \\
q \\
\end{pmatrix}  -
\begin{pmatrix}
 P \Lambda^{(n-n')}\cdot z' \\
R(\frac{\pi}{2}\cdot (m-m')) w_1' \\
\vdots \\
R\left(\frac{\pi}{2}\cdot (m-m') \right) w_{n-3}' \\
q' \\
\end{pmatrix} \\
t_0(n-n') \\
t_0(m-m') \\
\end{pmatrix} = \begin{pmatrix}
P(z -\Lambda^{(n-n')}\cdot z' )\\
w_1- R(\frac{\pi}{2}\cdot (m-m')) w_1' \\
\vdots \\
w_{n-3}-R(\frac{\pi}{2}\cdot (m-m')) w_{n-3}'  \\
q-q' \\
t_0(n-n') \\
t_0(m-m') \\
\end{pmatrix}
\end{split}
\]
\end{comment}
which is in $\Gamma$ as $\Lambda^{r}$  and $ R\left (\frac{\pi}{2} \cdot r \right)$ are integer matrix for each $r \in \Z$.\\
The action induced by  $\Gamma$ on $G^{2n}_{\overline{b},\frac{2\pi}{t_0},\frac{\pi}{2t_0}}$ is generated by 
\[
\begin{split}
&\gamma_1(x_1,x_2,\dots,x_{2n-2},x_{2n-1},x_{2n})=(x_{1}+p_{11}, x_{2}+p_{21},x_{3}+p_{31}, x_4,\dots,x_{2n-1},x_{2n}), \\
&\gamma_2(x_1,x_2,\dots,x_{2n-2},x_{2n-1},x_{2n})=(x_{1}+p_{12}, x_{2}+p_{22},x_{3}+p_{32}, x_4,\dots,x_{2n-1},x_{2n}), \\
&\gamma_3(x_1,x_2,\dots,x_{2n-2},x_{2n-1},x_{2n})=(x_{1}+p_{13}, x_{2}+p_{23},x_{3}+p_{33}, x_4,\dots,x_{2n-1},x_{2n}), \\
&\gamma_{l} (x_1,x_2,\dots,x_{2n-2},x_{2n-1},x_{2n})=(x_1,x_2,x_3,\dots,x_l+1,\dots,x_{2n-1},x_{2n}), \ \text{$l=4,\dots,2n-2$} \\
&\gamma_{2n-1}(x_1,x_2,\dots,x_{2n-2},x_{2n-1},x_{2n})=(A \cdot (x_1,x_2,x_3),x_4,\dots,x_{2n-2},x_{2n-1}+t_0,x_{2n}) \\
&\gamma_{2n}(x_1,x_2,\dots,x_{2n-2},x_{2n-1},x_{2n})=(x_1,x_2,x_3,\dots,\underbrace{x_{2j+1},-x_{2j}}_{\begin{subarray}{l}\text{entries\ $2j$ \ and \ $2j+1$}\\ \text{for each $j=2,\dots,n-2$}\end{subarray}},\dots,x_{2n-2},x_{2n-1},x_{2n}+t_0).
\end{split}
\]
It is immediate to see that the action is free and properly discontinuous, which implies that $\Gamma$ is a lattice of $G^{2n}_{\overline{b},\frac{2\pi}{t_{0}},\frac{\pi}{2t_{0}}}$. 
We compute now the commutators $[\gamma_r,\gamma_s]$ of $\Gamma$, for each pair $r,s \in \{1,\dots,2n\}$. 
The only non-trivial commutators are 
\[
\begin{split}
&[\gamma_{2n-1},\gamma_i]=\gamma_1^{\Lambda_{1i}}\cdot \gamma_2^{\Lambda_{2i}}\cdot \gamma_3^{\Lambda_{3i}}\cdot  \gamma_i^{-1}, \  i=1,\dots,3, \quad 
[\gamma_{2n},\gamma_{2l}]=\gamma_{2l}^{-1} \cdot \gamma_{2l+1}^{-1}, \  l=2,\dots,n-2,  \\
&[\gamma_{2n},\gamma_{2l+1}]=\gamma_{2l} \cdot \gamma_{2l+1}^{-1}, \  l=2,\dots,n-2, \\
\end{split}
\]
where $(\Lambda_{ij})$ is the integer matrix described previously. Then, since $\Gamma$ is $2$-step solvable, it follows that $[\Gamma, \Gamma]$ is a torsion-free abelian subgroup of $\Gamma$ of rank $2n-3$. Therefore, by 
\[
\rank(\Gamma)=\rank(\Gamma/[\Gamma, \Gamma])+\rank([\Gamma, \Gamma]),
\]
$\Gamma/[\Gamma, \Gamma]=\Z^3$ and, accordingly, $b_1(M)=3$.
\end{proof}
\begin{remark}
Observe that since the group $\Gamma \backslash G^{2n}_{\overline{b},\frac{2\pi}{t_{0}},\frac{\pi}{2t_{0}}}$ is not completely solvable, we cannot apply Hattori's Theorem \cite{HA} to compute the De Rham cohomology of $M$. 
\end{remark}


\begin{thebibliography}{999}
\bibitem{ADE} Alekseevsky, D.V.  and David,  L.,  A note about invariant SKT structures and generalized K\"ahler structures on flag manifolds,   Proc. Edinb. Math. Soc.  (2012), 543-549.
\bibitem{AO} Andrada, A. and Origlia, M., Lattices in almost abelian Lie groups with locally conformal K\"ahler or symplectic structures.  Manuscripta Math. 155 (2018), no.3-4,  389--417.
\bibitem{AOUV} Angella, D., Otal, A., Ugarte, L. and Villacampa, R., Complex structures of splitting type. Rev. Mat. Iberoam.  33(4), (2017), 1309-1350.
\bibitem{AGG}  Apostolov, V.,  Gauduchon, P.  and  Grantcharov, G.,  Bihermitian structures on complex surfaces,  Proc. London Math. Soc.(3) 79  (1999), 414--428. Corrigendum  92. (2006), no. 1, 200--202.
\bibitem{AG} Apostolov, V.  and Gualtieri, M., Generalized K\"ahler manifolds, commuting complex structures, and split tangent bundles,  Comm. Math. Phys. 27 (2007), no. 2, 561--575.

\bibitem{AL} Arroyo, R.M. and Lafuente, R.A.,  The long‐time behavior of the homogeneous pluriclosed flow,  Proc. Lond. Math. Soc. (3) 119(1) (2019), 266--289.
\bibitem{ArNic} Arroyo, R.M. and Nicolini, M.,  SKT structures on nilmanifolds,  Math. Z.  302(2) (2022), 1307--1320.

\bibitem{BE} Belgun, F.A.,  On the metric structure of non-K\"ahler complex surfaces, Math. Ann. 317 (2000), no.1, 1–40.
\bibitem{Bi}  Bismut,J. M. ,   A local index theorem for non Kähler manifolds,  Math.  Ann.  284(4) (1989), 681--699.
\bibitem{BC} Bock, C., On low-dimensional solvmanifolds, Asian J. Math. 20 (2016), no.2,  199--262.
\bibitem{BF}Brienza, B. and Fino, A. Generalized K\" ahler manifolds via mapping tori,  preprint arXiv:2305.11075, to appear in J. Symplectic Geom.
\bibitem{BM}  Boucetta,  M., and  Mansouri, W., Left invariant generalized complex and K\"ahler structures on simply connected four dimensional Lie groups: classification and invariant cohomologies,   J. Algebra 576 (2021), 27--94. 
\bibitem{CZ} Cao, K. and Zheng, F.,  Fino–Vezzoni conjecture on Lie algebras with abelian ideals of codimension two,  Math.  Z. 307  (2024), no. 2, Paper No. 31.
\bibitem{CA} Cavalcanti, G.R., Formality in generalized Kähler geometry, Topology Appl. 154 (2007), no.6,  1119--1125.
\bibitem{CH} S. S. Chern,  Complex manifolds without potential theory, Universitext Springer-Verlag, New York-Heidelberg, 1979, iii+152 pp.
\bibitem{DM} J. Davidov,  O. Mushkarov,  Twistorial construction of generalized Kähler manifolds,   J. Geom. Phys. 57(2007), 889--901.
\bibitem{DF} Dotti, I.G., and Fino, A.,  HyperKähler torsion structures invariant by nilpotent Lie groups. Classical Quant. Grav. 19, (2002), no. 3,  551–562.
\bibitem{EFV} Enrietti, N., Fino, A. and Vezzoni, L., Tamed symplectic forms and SKT metrics, J. Symplectic Geom. 10 (2012), no. 2, 203--223.
\bibitem{FG} Fino, A. and Grantcharov, G.,  Properties of manifolds with skew-symmetric torsion and special holonomy, Adv. Math. 189 (2004), no. 2, 439--450.
\bibitem{FP2} Fino, A. and Paradiso, F., Generalized K\"ahler almost abelian Lie groups, Ann. Mat. Pura Appl. (4) 200 (2021), no. 4, 1781--1812.
\bibitem{FP} Fino, A. and Paradiso, F., Hermitian structures on a class of almost nilpotent solvmanifolds, J. Algebra 609 (2022),  861--925.
\bibitem{FP3} Fino, A. and Paradiso, F., Hermitian structures on six-dimensional almost nilpotent solvmanifolds, preprint arXiv:2306.03485.
\bibitem{FP4} Fino, A. and Paradiso, F., Balanced Hermitian structures on almost abelian Lie algebras. J. Pure Appl. Algebra,(2023) 227(2), p.107186.
\bibitem{FPS04} Fino, A.,  Parton, M., and Salamon, S., Families of strong KT structures in six dimensions, 
Comment. Math. Helv. 79 (2004), no. 2, 317--340.
\bibitem{FT} A. Fino, A. Tomassini, Non-K\"ahler solvmanifolds with generalized K\"ahler structure, J. Symplect.Geom. 7 (2009), 1--14.
\bibitem{FV} Fino, A. and Vezzoni, L.,  On the existence of balanced and SKT metrics on nilmanifolds,  Proc. Amer. Math. Soc., 144(6), (2016),  2455--2459.
\bibitem{FS} Freibert, M. and Swann, A. Two-step solvable SKT shears,  Math. Z. 299, (2021), no. 3, 1703-1739.
\bibitem{FS2} Freibert, M. and Swann, Compatibility of balanced and SKT metrics on two-step solvable Lie groups,  preprint arXiv:2203.16638,  to appear in Transform. Groups.
\bibitem{PG} Gauduchon, P.,  Hermitian connections and Dirac operators, Boll. Un. Mat. Ital. B (7)11(1997), no. 2, 257--288.
\bibitem{Gualtieri} Gualtieri, M.,  Generalized complex geometry, DPhil thesis, Oxford University, 2004.
\bibitem{CG}   Cavalcanti, G. R.  and  Gualtieri, M., Blowing up generalized K\"ahler $4$-manifolds, Bull. Braz. Math. Soc. (N.S.)
 42 (2011),  537--557.
\bibitem {GZ}  Guo, Y. and Zheng, F.,  Hermitian geometry of Lie algebras with abelian ideals of codimension 2, Math. Z. 304 (2023), no. 3, Paper No. 51, 24 pp.
\bibitem{Has} Hasegawa, K., Complex and K\"ahler structures on compact solvmanifolds, J. Symplectic Geom. 3 (2005), no. 4, 749--767.
\bibitem{HA} Hattori, A.,Spectral sequence in the de Rham cohomology of fibre bundles, J. Fac. Sci. Univ. Tokyo Sect. I 8, (1960),  289–331.
\bibitem{Hitchin} Hitchin, N. J.,  Instantons, Poisson structures and generalized K\"ahler geometry, Comm. Math. Phys.  265 (2006), 131-164.
\bibitem{LRV} Lauret, J. and Rodríguez Valencia, E.A., On the Chern‐Ricci flow and its solitons for Lie groups,  Math. Nachr. 288  (2015), 1512--1526.
\bibitem{MS} Madsen, T.B. and Swann, A., Invariant strong KT geometry on four-dimensional solvable Lie groups,  J. Lie Theory, 21 (2011), 55--70.
\bibitem {RT} Rawashdeh, M. and Thompson, G., The inverse problem for six-dimensional codimension two nilradical Lie algebras, J. Math. Phys. 47 (2006), no. 11, 112901, 29 pp.
\bibitem{SR} Rollenske, S., Dolbeault cohomology of nilmanifolds with left-invariant complex structure, in  W. Ebeling, K. Hulek, K. Smoczyk (eds.), Complex and Differential Geometry: Conference held at Leibniz Universität Hannover, September 14--18, 2009, Springer Proceedings in Mathematics 8, Springer, 2011, 369–392 	
\bibitem{Strominger}  Strominger, A.,  Superstrings with torsion, Nuclear Phys. B 274 (1986), 254--284.
\bibitem{UG} Ugarte, L., Hermitian structures on six-dimensional nilmanifolds,  Transform. Groups 12 (2007), 175-202.
\bibitem{VZ} Vezzoni, L., A note on canonical Ricci forms on 2-step nilmanifolds, Proc. Amer. Math. Soc. 141 (2013), 325-333.
\end{thebibliography}
\end{document}